\documentclass[12pt,a4paper]{amsart}
\usepackage{tikz-cd}
\usepackage{amsfonts}
\usepackage{amsthm}
\usepackage{amsmath}
\usepackage{amscd}
\usepackage{calc}
\usepackage[latin2]{inputenc}
\usepackage{t1enc}
\usepackage{indentfirst}
\usepackage{graphicx}
\usepackage{graphics}
\usepackage{pict2e}
\usepackage{epic}
\numberwithin{equation}{section}
\usepackage[margin=2.9cm]{geometry}
\usepackage{epstopdf} 
\usepackage{mathrsfs}
\usepackage{enumitem}
\usepackage{amsrefs}

\theoremstyle{plain}
\newtheorem{Th}{Theorem}[section]
\newtheorem{Lemma}[Th]{Lemma}

\newtheorem{Prop}[Th]{Proposition}

\theoremstyle{definition}

\newtheorem{Rem}[Th]{Remark}
\newtheorem{?}[Th]{Problem}

\newcommand{\im}{\operatorname{im}}

\newcommand{\R}{\mathbb{R}}

\renewcommand{\epsilon}{\varepsilon}

\newcommand{\abs}[1]{\lvert #1 \rvert}

\newcommand{\lpnorm}[3]{\abs{\abs{#1}}_{L^{#2}(#3)}}

\title{Concentration of quantum integrable eigenfunctions on a convex surface of revolution}
\author[M. Geis]{Michael Geis}
\address{Department of Mathematics, Northwestern University, 2033 Sheridan Road, Evanston, IL 60208}
\email{mlg@math.northwestern.edu}
\thanks{Partially supported by NSF RTG grant DMS-1502632.}
\thanks{Written in progress of a PhD in mathematics at Northwestern University}
\date{August 27, 2020}

\begin{document}

\maketitle

\begin{abstract}
Let $(S^2,g)$ be a convex surface of revolution and $H \subset S^2$ the unique rotationally invariant geodesic. Let $\varphi^\ell_m$ be the orthonormal basis of joint eigenfunctions of $\Delta_g$ and $\partial_\theta$, the generator of the rotation action. The main result is an explicit formula for the weak-* limit of the normalized empirical measures, $\Sigma_{m = -\ell}^\ell \lpnorm{\varphi^\ell_m}{2}{H}^2 \delta_{\frac{m}{\ell}}(c)$ on $[-1,1]$. The explicit formula shows that, asymptotically, the $L^2$ norms of restricted eigenfunctions are minimal for the zonal  eigenfunction $m = 0$, maximal for Gaussian beams $m = \pm 1$, and exhibit a $(1 - c^2)^{-\frac{1}{2}}$ type singularity at the endpoints. For a pseudo-differential operator $B$ we also compute the limits of the normalized measures $\sum_{m = -\ell}^\ell \langle B \varphi^\ell_m , \varphi^\ell_m \rangle \delta_{\frac{m}{\ell}}$.
\end{abstract}

\section{Introduction}

This article is concerned with concentration properties of an orthonormal basis of Quantum completely integrable Laplace eigenfunctions 

\begin{equation}
    -\Delta_g \varphi_{\lambda_j} = \lambda^2_j \varphi_{\lambda_j}
\end{equation}
 on a closed Riemannian manifold $(M,g)$ in the $\lambda_j \to \infty$ limit. The concentration of a sequence $\varphi_{\lambda_j}$ of eigenfunctions is often measured by studying the limits of matrix elements $\langle A \varphi_{\lambda_j} , \varphi_{\lambda_j} \rangle_{L^2(M)}$ of pseudodifferential operators, which are known as microlocal defect measures. One may also study concentration on a submanifold $H \subset M$ via the limits of $L^p$ norms of restricted eigenfunctions $\lpnorm{\varphi_{\lambda_j}|_{H}}{p}{H}$. We take a new approach to the study of eigenfunction concentration in the quantum completely integrable setting in the simple case of a convex surface of revolution $(S^2,g)$ where we can obtain explicit results. We let $\partial_\theta$ be the smooth vector field which generates the $S^1$ symmetry and we study an $L^2$ orthonormal basis of joint eigenfunctions  $\varphi^\ell_m$ of the commuting operators $-\Delta_g$ and $D_\theta = \frac{1}{i}\partial_\theta$:

$$\begin{cases}
    -\Delta_g\varphi^\ell_m = \lambda_\ell^2 \varphi^\ell_m \\
    D_\theta \varphi^\ell_m = m\varphi^\ell_m
\end{cases}$$

On a convex surface of revolution, there exists an operator $\widehat{I}_2$ which commutes with $-\Delta_g$ and $D_\theta$ which has a joint spectrum with $D_\theta$ consisting of a lattice of simple eigenvalues

$$\text{Spec}(\widehat{I_2}, D_\theta) = \{ (\ell,m) \in \mathbb{Z}^2 ~|~ \ell \geq 0 ; \abs{m} \leq \ell  \}$$

Thus $\widehat{I}_2 \varphi^\ell_m = \ell \varphi^\ell_m $, $D_\theta \varphi^\ell_m = m \varphi^\ell_m$. Our purpose is to study the relative rates of concentration of the eigenfunctions $\varphi^\ell_m$ along the equator, $H$, the unique rotationally invariant geodesic, within a single $\widehat{I}_2$ eigenspace. To do this we calculate the weak-* limits of the following $\textit{empirical measure}$:

\begin{equation} \label{eq:defnmuell}
    \mu_\ell = \frac{1}{M_\ell} \sum_{m = - \ell}^{\ell} \lpnorm{\varphi^\ell_m}{2}{H}^2  \delta_{\frac{m}{\ell}} 
\end{equation}

 The constant $M_\ell$ normalizes $\mu_\ell$, making it a probability measure on $[-1,1]$. If we view $c \in [-1,1]$ as a continuous version of the ratio $m/\ell$, the weak-* limits can be viewed as the asymptotic distribution of mass across the $\widehat{I}_2$ eigenspaces. In the calculation of the limit of $\mu_\ell$, we need to compute the weak-* limit of another family of empirical measures

\begin{equation} \label{eq:defnnuell}
    \nu_\ell(B) = \frac{1}{N_\ell(B)} \sum_{m = -\ell}^\ell \langle B\varphi^\ell_m , \varphi^\ell_m \rangle_{L^2(S^2,dV_g)} \delta_{\frac{m}{\ell}}
\end{equation}

Here $B \in \Psi^0$ is a homogeneous pseudo-differential operator of order zero.

\subsection{Statement of results}

In order to state the results, we need to briefly describe the underlying geometry. The principal symbols of $\widehat{I}_2$ and $\widehat{I}_1 = D_\theta$, $I_2$ and $I_1 = p_\theta$ are homogeneous, poisson commuting smooth functions on $T^*S^2 \setminus 0$ and are called the action variables for the geodesic flow. Their Hamiltonian flows are $2\pi$-periodic so their joint flow $\Phi_{\mathbf{t}}$ defines a homogeneous, Hamiltonian action of the torus $T^2$. The joint flow preserves level sets of both $I_2$ and $p_\theta$ and by homogeneity, all of the information is contained in the $I_2 = 1$ level set, which we denote by $\Sigma \subset T^*S^2 \setminus 0$. On $\Sigma$, $\abs{I_1} \leq 1$ and for $c \in [-1,1]$, we let $T_c = I_1^{-1}(c) \cap \Sigma$. For $c \neq \pm 1$, these level sets are diffeomorphic to $T^2$ and consist of a single orbit of the joint flow. The levels $T_{\pm 1}$ consist of $I_2$ unit covectors tangent to $H$ with the sign reflecting the orientation relative to $\partial_\theta$. We let $d\mu_L$ denote Liouville measure on $\Sigma$ and $d\mu_{c,L} = d\mu_L/dp_\theta$ denote Liouville measures on the regular tori $T_c$. The torus action $\Phi_{\mathbf{t}}$ commutes with the geodesic flow $G^t = \exp tH_{\abs{\xi}_g}$ and we can write 

\begin{equation} \label{eq:defnK}
    \abs{\xi}_g = K(I_1,I_2)
\end{equation}

For a smooth function $K$ on $\R^2\setminus 0$, homogeneous of degree 1. We let $(\omega_1,\omega_2) = \nabla_I K(I_1,I_2)$ be the so-called frequency vector associated to this action. The $\omega_i$ are themselves functions of the action variables $I_1,I_2$. For a homogeneous pseudo $B \in \Psi^0(S^2)$, we let $\sigma(B)$ denote its principal symbol and set $\widehat{\sigma(B)}(c) = \int_{T_c} \sigma(B) \, d\mu_{c,L}$. We also let $\omega(B) = \int_\Sigma \sigma(B) d\mu_L$  be the Liouville state on $B$. We note that for $(x,\xi) \in T_H S^2 \cap T_c$ we have,

$$p_\theta(x,\xi)^2 = \abs{\xi}_g^2 a(r_0)^2\cos^2 \phi = K(c,1)^2a(r_0)^2\cos^2\phi$$

where $\phi$ is the angle between the covector $\xi$ and $H$ and $r_0$ is the distance from the north pole to $H$ so that $H = \{ r = r_0 \}$. Let $\mathscr{L}(H)$ be the length of $H$. Then $a(r_0) = \mathscr{L}(H)/2\pi$.

\begin{Th}

Let $(S^2,g)$ be a convex surface of revolution where $g = dr^2 + a(r)^2d\theta^2$ in geodesic polar coordinates. Let $H \subset S^2$ be the equator, the unique rotationally invariant geodesic. Then in terms of action angle variables we have,

\leavevmode

\begin{enumerate}[label = (\alph*)]

    \item For every $f \in C^0([-1,1])$,
    
$$
\int_{-1}^{1} f(c) \, d\mu_\ell(c) = \frac{1}{M_\ell} \sum_{m = -\ell}^\ell \lpnorm{\varphi^\ell_m}{2}{H}^2 f\left(\frac{m}{\ell}\right)  \to \frac{1}{M} \int_{-1}^{1} f(c) \frac{\omega_2(c,1)}{\sqrt{1 - \frac{(2\pi)^2c^2}{K(c,1)^2\mathscr{L}(H)^2}}} \,dc
$$

    \item For any $f \in C^0([-1,1])$, 
    
$$
\int_{-1}^{1} f(c) \, d\nu_\ell(c) = \frac{1}{N_\ell(B)} \sum_{m = -\ell}^\ell \langle B\varphi^\ell_m , \varphi^\ell_m \rangle_{L^2(S^2,g)} f\left(\frac{m}{\ell} \right) \to \frac{1}{\omega(B)}\int_{-1}^{1} f(c) \widehat{\sigma(B)}(c) \,dc
$$

\end{enumerate}

The constant appearing in (a) is 

$$M = \int_{-1}^1 \frac{\omega_2(c,1)}{\sqrt{1 - \frac{(2\pi)^2c^2}{K(c,1)^2\mathscr{L}(H)}}} \,dc $$

and normalizes the limit measure to have mass 1 on $[-1,1]$. We note that when $c = \pm 1$, $T_c$ collapses to the set of $I_2$ unit covectors tangent to $H$. On $T_{\pm 1}$, $\omega_2 = \frac{\partial K}{\partial I_2} = \frac{\partial K}{\partial I_1} = a(r_0)^{-2} = \frac{(2\pi)^2}{\mathscr{L}^2(H)}$ and $\phi = 0$. The left hand of \eqref{eq:limitmuell} therefore blows up at $c = \pm 1$.

\end{Th}

When $(S^2,g_{can})$ is the standard sphere, $\mathscr{L}(H) = 2\pi$, $K(c,1) = 1$ and $\omega_2(c,1) = 1$, hence 

\begin{equation} \label{eq:limitmuell}
    \frac{\omega_2(c,1)}{\sqrt{1 - \frac{(2\pi)^2c^2}{K(c,1)^2\mathscr{L}(H)^2}}} = \frac{1}{\sqrt{1 - c^2}}
\end{equation}

\begin{Rem}
It would be interesting to know when the above formula holds. It is plausible that this is true on an ellipsoid of revolution where one has explicit formulae for the frequencies $\omega_i$. We would also like to find the weak-* limit of the measure \eqref{eq:defnmuell} when $H$ is any latitude circle. We leave that for future investigations. 

\end{Rem}

The measures $\mu_\ell$ considered here are closely related to the empirical measures associated to a polarized toric K\"ahler manifold $L \to M^n$ studied in \cite{ZZ},

$$
\mu^z_k = \frac{1}{\Pi_{h^k}(z,z)} \sum_{\alpha \in kP \cap \mathbb{Z}^d} \abs{s_\alpha(z)}^2_{h^k} \delta_{\frac{\alpha}{k}}
$$

Here, $s_\alpha(z)$ are the holomorphic sections of $L^k$. These correspond to lattice points inside the $k^{th}$ dialate of a certain Delzant polytope $P \subset \R^n$. This polytope is the image of the moment map $\mu : M \to P$ associated to the torus action on $M$. In our setting, $M$ is analogous to the phase space energy surface $\Sigma = \{ I_2 = 1 \}  \subset T^*S^2$ with the moment map $I_1 : \Sigma \to [-1,1]$. The joint eigenfunctions of $\widehat{I}_2$-eigenvalue $\ell$ correspond to the lattice points inside the $\ell^{th}$ dialate of $I_1(\Sigma) = [-1,1]$ and are analogous to the holomorphic sections $s_\alpha$. In both cases the measures are dialated back to be supported on the image of the moment map and normalized to have mass 1. The submanifold $H$ plays the role of the continuous parameter $z \in M$ in the K\"ahler setting. In \cite{ZZ} it is shown that as $k \to \infty$, a central limit theorem type rescaling of these measures tends to a Gaussian measure centered on $\mu(z)$ while in our case the measures $\mu_\ell$ tend to an absolutely continuous limit which blows up at the end points with a $(1 - c^2)^{-\frac{1}{2}}$ type singularity. The blow-up reflects the fact that the Gaussian beams $m = \pm \ell$ are concentrated on $T^*H \cap \Sigma$ in phase space.

\vspace{5mm}

In addition, we codify the similarity of the operator $\widehat{I}_2$ on $(S^2,g)$ to the degree operator $A = \sqrt{-\Delta_{g_{can}} + \frac{1}{4}} - \frac{1}{2}$ on the round sphere $(S^2,g_{can})$ by showing that $\widehat{I}_2$ and $A$ are conjugate via a unitary Fourier integral operator that leaves invariant $D_\theta$, at least up to a finite rank operator. 

\begin{Th}
Let $(S^2,g)$ be a convex surface of revolution and $A = \sqrt{-\Delta_{g_{can}} + \frac{1}{4}} - \frac{1}{2}$ be the degree operator on the round sphere. There exists a homogeneous unitary Fourier integral operator $$W : L^2(S^2,g_{can}) \to L^2(S^2,g)$$ such that $[W,D_\theta] = 0$ and $W^* \widehat{I}_2 W = A + R$ where $R$ is a finite rank operator. Consequently, if $Y^\ell_m$ denotes the standard orthonormal basis of $L^2(S^2,g_{can})$ such that $AY^\ell_m = \ell Y^\ell_m$, $D_\theta Y^\ell_m = m Y^\ell_m$, then for $\ell$ large enough, there are constants $c^\ell_m$ with $\abs{c^\ell_m} = 1$ so that

\begin{equation}
WY^\ell_m = c^\ell_m \varphi^\ell_m
\end{equation}

\end{Th}

In \cite{L}, Lerman proves that there is only one homogeneous hamiltonian action of the torus $T^2$ on $T^*S^2 \setminus 0$ up to symplectomorphism. In particular, letting $p_2(x,\xi) = \abs{\xi}_{g_{can}(x)}$ be the principal symbol of $A$, $p_\theta$ and $p_2$ generate such an action, so there is a homogeneous symplectomorphism $\chi$ on $T^*S^2 \setminus 0$ which pulls back the functions $p_\theta, I_2$ to $p_\theta$, $p_2$. Theorem 1.2 is essentially an operator theoretic version of this statement.

\subsection*{Acknowledgements}
I would like to thank Emmett Wyman for many helpful conversations regarding the symbol calculus of FIOs, as well as my advisor Steve Zelditch for his continual patience and guidance.

\subsection{Outline of the computation of weak-* limits}

We compute the weak-* limits of the measures \eqref{eq:defnnuell}, \eqref{eq:defnmuell} by expressing their un-normalized versions as a trace and using the symbol calculus of Fourier integral operators to compute the leading order contribution as $\ell \to \infty$. We refer to \cite{LPDO4},\cite{GuS} for background on Fourier integral operators and the symbol calculus. Let $\Pi_\ell : L^2(S^2,dV_g) \to L^2(S^2,dV_g)$ denote the orthogonal projection onto the $\widehat{I}_2 = \ell$ eigenspace. Suppose that $A : C^\infty(S^2) \to C^\infty(S^2)$ is an operator which commutes with $D_\theta$. Then the kernel of the operator 

\begin{equation}
    f \left( \frac{D_\theta}{\ell} \right) A\Pi_\ell
\end{equation}

is equal to 

\begin{equation}
    \sum_{m = -\ell}^\ell A\varphi^\ell_m(x)\overline{\varphi^\ell_m(y)} f\left( \frac{m}{\ell} \right)
\end{equation}

And thus we have 

\begin{equation} \label{eq:traceformulaintro}
    \text{Trace} \, f \left( \frac{D_\theta}{\ell} \right) A\Pi_\ell = \sum_{m = -\ell}^\ell \langle A\varphi^\ell_m , \varphi^\ell_m \rangle f\left( \frac{m}{\ell}\right)
\end{equation}

We use this formula to compute the weak-* limits of both sequences of empirical measures. When $A$ is a pseudo-differential operator, this forumula returns the unnormalized measures \eqref{eq:defnnuell} tested against $f$. To use this formula for the measures $\eqref{eq:defnmuell}$, we express the $L^2$ norms on $H \subset S^2$ as a global matrix element as follows: let $\gamma_H : C^\infty(S^2) \to C^\infty(H)$ denote restriction to $H$ and $\gamma^*_H$ denote the $L^2$ adjoint of $\gamma_H$ with respect to the Riemannian volume measure $dV_g$. Thus, for $g \in C^\infty(H)$, $f \in C^\infty(S^2)$ we have 

$$\langle \gamma^*_H g , f \rangle_{L^2(S^2,dV_g)} = \int_H  g f|_{H} \, dS$$

where $dS$ is the induced surface measure. From this it follows that 

$$\lpnorm{\varphi^\ell_m}{2}{H,dS}^2 = \langle \gamma_H^*\gamma_H \varphi^\ell_m , \varphi^\ell_m \rangle$$

One problem with this setup is that \eqref{eq:traceformulaintro} requires the operator $A$ to commute with $D_\theta$, and this will not be true for every pseudo $B \in \Psi^0(S^2)$ nor for the operator $\gamma_H^*\gamma_H$. We deal with this by averaging against the torus action generated by $D_\theta$ and $\widehat{I}_2$. For $\mathbf{t} = (t_1,t_2) \in T^2$, let

\begin{equation}
    U(\mathbf{t}) = \exp i[t_1 D_\theta + t_2 \widehat{I}_2]
\end{equation}

 In section 3 we review that this is a torus action on $L^2(S^2, dV_g)$ by unitary Fourier integral operators. For any operator $A : C^\infty(S^2) \to C^\infty(S^2)$ we set 
 
 \begin{equation} \label{eq:torusavg}
     \bar{A} = (2\pi)^{-2} \int_{T^2} U(\mathbf{t})^*AU(\mathbf{t}) \, d\mathbf{t}
 \end{equation}

The average $\bar{A}$ commutes with both $D_\theta$ and $\widehat{I}_2$ since 

\begin{equation}
    [D_\theta, \bar{A}] = (2\pi)^{-2} \int_{T^2} -\partial_{t_1}[U(\mathbf{t})^*AU(\mathbf{t})] d\mathbf{t} = 0
\end{equation}

And similarly for $\widehat{I}_2$. We also note that

$$\langle A \varphi^\ell_m , \varphi^\ell_m  \rangle_{L^2(S^2), dV_g} = \langle \bar{A} \varphi^\ell_m , \varphi^\ell_m  \rangle_{L^2(S^2,dV_g)}$$

This means replacing $A$ with $\bar{A}$ in the trace will not change the right hand side of $\eqref{eq:traceformulaintro}$. When $A \in \Psi^0(S^2)$, Egorov's theorem tells us that $\bar{A} \in \Psi^0(S^2)$ as well, and 

$$\sigma(\bar{A}) = (2\pi)^{-2} \int_{T^2} \Phi_{\mathbf{t}}^*\sigma(A) \, d\mathbf{t}$$

where $\Phi_\mathbf{t}$ is the joint flow generated by $I_1 = p_\theta$ and $I_2$. In section 4, we analyze the averaged restriction operator

\begin{equation}
    \bar{V} = (2\pi)^{-2} \int_{T^2} U^*(\mathbf{t})(\gamma_H^*\gamma_H)U(\mathbf{t}) \, d\mathbf{t}
\end{equation}

And show that, after applying microlocal cutoffs to $\gamma^*_H\gamma_H$, it splits into the sum of a pseudo-differential operator and a Fourier integral operator. The canonical relation of the non-pseudo-differential part of $\bar{V}$ is related to the notion of a mirror reflection map on covectors based on $H$ (See section 4 for details). Both summands can be made to commute with $U(\mathbf{t})$. The strategy of using the operator $\bar{V}$ to study restricted $L^2$ norms (and more generally restricted $\Psi$DO matrix elements) has been used in \cite{TZ1} and we closely follow their analysis here. As mentioned, for this analysis to work we need to microlocally cut off $\gamma_H^*\gamma_H$ away from both $N^*H$ and $T^*H$. Literally speaking we fix $\epsilon > 0$ and instead work with the operator

\begin{equation}
    (\gamma_H^*\gamma_H)_{\geq \epsilon} = (1 - \widehat{\chi}_{\epsilon/2})(\gamma_H^*\gamma_H)(1 - \widehat{\chi}_\epsilon)
\end{equation}

Where $(I - \widehat{\chi}_\epsilon)$ is a homogeneous pseudo-differential operator with wave front set outside conic neighborhoods of both $N^*H$ and $T^*H$. The cutoff away from the normal directions is technical and related to the choice to use the homogeneous calculus, while the cutoff away from the tangential directions is necessary since otherwise the canonical relation of $\bar{V}$ would be singular. We show in section 5 that we can use the cutoff operator $(\gamma_H^*\gamma_H)_{\geq \epsilon}$ to compute the weak-* limits of \eqref{eq:defnmuell} by letting $\epsilon \to 0$ afterwards.

\section{Quantum toric integrability for convex surfaces of revolution}

Let $(S^2,g)$ be a surface of revolution. We denote the two fixed points of the $S^1$ action by $N$ and $S$. Fix a meridian geodesic $\gamma_0$ which joins $N$ to $S$ and let $(r,\theta)$ denote geodesic polar coordinates from $N$, i.e. so that the curve $r \mapsto (r,0)$ is the arc length parametrized geodesic $\gamma_0$. In these coordinates the metric takes the form 

$$g = dr^2 + a(r)^2 d\theta^2$$

for some smooth function $a : [0, L] \to \R{}$ such that $a^{2k}(0) = a^{2k}(L) = 0$ and $a'(0) = 1$, $a'(L) = 1$. Here $L$ is the distance between the poles. A convex surface of revolution is one such that $a(r)$ has exactly one non-degenerate critical point which is a maximum, $a''(r_0) < 0$. The latitude circle $H = \{ (r = r_0) \}$ is the unique rotationally invariant geodesic.

\vspace{5mm}

Recall that we say the Laplacian $-\Delta_g$ of a Riemannian manifold $(M^n,g)$ is quantum completely integrable if there exists $n$ first order homogeneous pseudo-differential operators $P_1, \dots, P_n \in \Psi^1(M)$ satisfying:

\begin{itemize}
    \item $[P_i , P_j = 0]$
    \item $\sqrt{-\Delta_g} = K(P_1, \dots , P_n)$ for some polyhomogeneous function $K \in C^\infty(\R^n \setminus 0)$
    \item If $p_j = \sigma(P_j)$ are the principal symbols, the regular values of the associated moment map $\mathcal{P} = (p_1,\dots,p_n) : T^*M \setminus 0 \to \R^n \setminus 0$ form an open, dense subset of $T^*M$.
\end{itemize}

For background on quantum integrable Laplacians, see chapter 11 of \cite{cbms}. If $(S^2,g)$ any surface of revolution, and $D_\theta = \frac{1}{i} \partial_\theta$ is the self-adjoint differential operator associated to the generator of the $S^1$ action, it is clear by writing $\Delta_g$ in polar coordinates that $[\Delta_g , D_\theta] = 0$. Hence every surface of revolution is quantum completely integrable by taking $P_1 = \sqrt{-\Delta_g}$ and $P_2 = D_\theta$. The third condition is satisfied, for instance, if $a(r)$ is assumed to be Morse. In the special case of a convex surface of revolution, Colin de Verdi\`ere in \cite{CdV} has shown that the Laplacian is quantum toric completely integrable. This means that there exists $\widehat{I}_1 , \widehat{I}_2$ first order, homogeneous, commuting pseudo-differential operators satisfying the above conditions of quantum complete integrability, but with the additional property that 

\begin{equation} \label{eq:toricint}
\exp 2\pi i \widehat{I}_j = \text{Id}
\end{equation}

In particular, one can take $\widehat{I}_1 = D_\theta$ and $\widehat{I}_2$ to be self-adjoint and elliptic. Note that condition \eqref{eq:toricint} implies that the joint spectrum of $\widehat{I}_1,\widehat{I}_2$ is a subset of $\mathbb{Z}^2$. In fact it is shown in \cite{CdV} that it consists of all simple eigenvalues and 

\begin{equation}
    \text{Spec}(\widehat{I}_1,\widehat{I}_2) = \{(m,\ell) \in \mathbb{Z}^2 ~|~ \abs{m} \leq \ell ; \ell > 0 \}
\end{equation}

We fix a particular orthonormal basis of joint eigenfunctions $\{\varphi^\ell_m\}$ satisfying $\widehat{I}_2 \varphi^\ell_m = \ell \varphi^\ell_m$ and $D_\theta \varphi^\ell_m = m \varphi^\ell_m$.

\subsection{The moment map and classical toric integrability}

Let $I_j = \sigma(\widehat{I}_j)$ be the principal symbols. The associated moment map $\mathcal{P} = (I_1,I_2) : T^*S^2 \setminus 0 \to \R^2 \setminus 0$ has image equal to the closed conic wedge 

$$\mathcal{B} = \{(x,y) ~|~ \abs{x} \leq y ; y > 0 \}$$

The set of critical points, $Z$, of $\mathcal{P}$ consists of covectors lying tangent to the equator. If $(\rho,\eta)$ are the dual coordinates to $(r,\theta)$ on the fibers of $T^*S^2$,

$$Z = \{(r_0,\theta,0,\eta) ~|~ \eta \neq 0 \} = T^*H \setminus 0$$

$\mathcal{P}$ maps $Z$ to the boundary $\partial \mathcal{B}$, so the interior of $\mathcal{B}$ consists entirely of regular values. Consider a regular level set of the form $T_c = \mathcal{P}^{-1}(1,c)$, for $c \in (-1,1)$. By homogeneity, all other regular levels are dialates of these. For each $c$, $T_c$ is connected and diffeomorphic to a torus $T^2 \cong \R/2\pi \mathbb{Z} \times \R/2\pi \mathbb{Z}$. The singular levels correspond to $c = \pm 1$ and are equal to the set of covectors $T_{\pm 1} = \{(r_0,\theta,0,\pm 1)\}$. One consequence of quantum toric integrability is of course classical toric integrability. That is, letting $H_{I_j}$ denote the hamilton vector fields of $I_j$, equation $\eqref{eq:toricint}$ implies that both $H_{I_j}$ generate $2\pi$-periodic flows. Since $\{I_1 , I_2 \} = 0$, we let for $\mathbf{t} = (t_1,t_2) \in T^2$,

\begin{equation}
    \Phi_\mathbf{t} : T^2 \times T^*S^2 \setminus 0 \to T^*S^2 \setminus 0
\end{equation}

$$\Phi_\mathbf{t}(x,\xi) = \exp t_1 H_{I_1} \circ \exp t_2 H_{I_2}(x,\xi)$$

The joint flow $\Phi_\mathbf{t}$ thus defines a homogeneous, Hamiltonian action of $T^2$ on $T^*S^2 \setminus 0$ which commutes with the geodesic flow $G^t = \exp tH_{\abs{\xi}_g}$. It preserves the level sets of the moment map and each torus $T_c$ consists of a single orbit of the joint flow. 

\subsection{The standard torus action on $T^*S^2$}

In \cite{L}, Lerman shows that up to symplectic equivalence, there is only one homogeneous Hamiltonian action of $T^2$ on $T^*S^2 \setminus 0$. The simplest example of a convex surface of revolution is the standard sphere $(S^2, g_{can})$. For the standard sphere we can take $\widehat{I}_2 = A = \sqrt{-\Delta_{g_{can}} + \frac{1}{4}} - \frac{1}{2}$, the so-called degree operator. The associated torus action on $T^*S^2$ is generated by $\abs{\xi}_{g_{can}}$ and $p_\theta$. If $I_1 = p_\theta$ and $I_2$ are the action variables associated to a convex surface of revolution, there is a homogeneous symplectomorphism 

$$\chi : T^*S^2 \setminus 0 \to T^*S^2 \setminus 0$$ 

such that $\chi^*p_\theta = p_\theta$ and $\chi^*I_2 = \abs{\xi}_{g_{can}}$. Theorem 1.2 is the statement that the symplectic equivalence of the torus action on a convex surface of revolution to that of the round sphere  can be quantized. That is, the generators of the standard torus unitary torus action $D_\theta$ and $A$ on the round sphere are unitarily conjugate via a homogeneous Fourier integral operator to the quantized action operators $\widehat{I}_j$ on any convex surface of revolution.

\section{The Quantum torus action}

In this section we briefly review the fact that the commuting operators $\widehat{I}_1 = D_\theta$ and $\widehat{I}_2$ on a convex surface of revolution $(S^2,g)$ together generate an action of $T^2$ on $L^2(S^2,dV_g)$ by unitary Fourier integral operators. (See for instance p. 245 of \cite{cbms}). For $\mathbf{t} = (t_1,t_2) \in T^2$ we set 

\begin{equation}
   U(\mathbf{t}) = \exp i[t_1 D_\theta + t_2 \widehat{I}_2] 
\end{equation}

\begin{Prop}
The operator $U(t_1,t_2)$ is a homogeneous Fourier integral operator belonging to the class $I^{-\frac{1}{2}}(T^2 \times S^2 \times S^2 ; C_U)$. Its canonical relation is given by the space-time graph of the joint flow

    $$C_U = \{(t_1,p_\theta(x,\xi), t_2 , I_2(x,\xi) , y , \eta , x, \xi) ~|~ (y,\eta) = \Phi_{(t_1,t_2)}(x,\xi) \, ; \, (x,\xi) \in T^*S^2 \setminus 0 \}$$

The half density part of the symbol $\sigma(U)$ pulls back along the parametrizing map 

$$\iota : (t_1,t_2,x,\xi) \mapsto (t_1,p_\theta(x,\xi), t_2 , I_2(x,\xi) ,\Phi_{(t_1,t_2)}(x,\xi), x, \xi)$$

to the half density $\abs{dt_1 \wedge dt_2}^\frac{1}{2} \otimes \abs{dx \wedge d\xi}^\frac{1}{2}$ on $T^2 \times T^*S^2$.

\end{Prop}

\begin{proof}

Since $\exp it_1D_\theta$ just acts by pulling back a function along the flow of the vector field $\partial_\theta$, one can check in coordinates that this is a Fourier integral operator in the class $I^{-\frac{1}{4}}(S^1 \times S^2, S^2 ; C) $ where 

$$C = \{t_1 , p_\theta(x,\xi), y,\eta,x,\xi) ~|~ (y,\eta) = \exp t_1 H_{p_\theta}(x,\xi) ; (x,\xi) \in T^*S^2\setminus 0 \}$$

The half density symbol pulls back along the parametrizing map 

$$\iota : (t_1 , x,\xi) \mapsto (t_1 , p_\theta(x,\xi), \exp t_1 H_{p_\theta}(x,\xi) , x, \xi)$$

to $\abs{dt_1}^\frac{1}{2} \otimes \abs{dx \wedge d\xi}^\frac{1}{2}$. Now $I_2$ is a first order, self-adjoint, elliptic pseudo-differential operator with integer spectrum, so by \cite{DG} we have that $\exp it_2\widehat{I}_2 \in I^{-\frac{1}{4}}(S^1 \times S^2 \times S^2; C' )$ where

$$C' = \{t_2 , I_2(x,\xi) , y,\eta,x,\xi) ~|~ (y,\eta) = \exp t_2H_{I_2}(x,\xi) ; (x,\xi) \in T^*S^2 \setminus 0 \}$$

Now the composition of $C$ with $C'$ is transverse since they are essentially canonical graphs. By standard transverse composition of FIOs the orders add and we get the description of $U(\mathbf{t})$ stated in the proposition.  

\end{proof}

\section{Restricted $L^2$ norms as matrix elements}

In order to calculate the weak-* limit of \eqref{eq:defnmuell} using trace formulae, we need to relate the restricted $L^2$ norms of the joint eigenfunctions to matrix elements. Let

$$\gamma_H : C^\infty(S^2) \to C^\infty(H)$$

be the operator which restricts functions to $H$. Then if $\gamma^*_H$ is the $L^2$ adjoint, we have

$$\lpnorm{\varphi^\ell_m}{2}{H}^2 = \langle \gamma^*_H\gamma_H \varphi^\ell_m , \varphi^\ell_m \rangle_{L^2(S^2,dV_g)}$$

and since $\varphi^\ell_m$ are joint eigenfunctions of $\widehat{I}_j$, we can replace $\gamma_H^*\gamma_H$ with the average 

$$\bar{V} = (2\pi)^{-2}\int_{T^2} U(\mathbf{t})^*(\gamma_H^*\gamma_H)U(\mathbf{t}) \, d\mathbf{t}$$

without changing the above matrix elements. The problem is the operator $\bar{V}$ has a singular canonical relation. To fix this, we replace $\gamma_H^*\gamma_H$ with a microlocally cut off operator $(\gamma_H^*\gamma_H)_{\geq \epsilon}$ described below. After doing this,

$$\bar{V}_\epsilon = (2\pi)^{-2}\int_{T^2} U(\mathbf{t})^*(\gamma_H^*\gamma_H)_{\geq \epsilon}U(\mathbf{t}) \, d\mathbf{t} $$

becomes a genuine Fourier integral operator and we calculate its order and symbolic data.

\subsection{The cutoff restriction operator on the sphere}

Let 

$$T^*_H S^2 = \{(x,\xi) \in T^*S^2 ~|~ x \in H \}$$ 

Denote the set of covectors with footprint on $H$. Since $\gamma_H$ is just pullback along the inclusion map, it is a Fourier integral operator associated with the pullback canonical relation 

$$C = \{(x,\xi|_{TH}, x, \xi)~|~ (x,\xi) \in T_H^*S^2 \setminus 0 \} \subset T^*H \times T^*S^2$$

The left factor contains elements of the zero section whenever $\xi \in N^*H$, so it is not a homogeneous Fourier integral operator in the sense of \cite{LPDO4}. Because of this defect, the wave front set of $\gamma_H^*\gamma_H$ is

\begin{equation}
    WF'(\gamma^*_H \gamma_H) = C_H \cup N^*H \times 0_{T^*M} \cup 0_{T^*M} \times N^*H
\end{equation}

Where $C_H \subset T^*M \setminus 0 \times T^*M \setminus 0$ is the homogeneous canonical relation

$$C_H = \{(x,\xi,x,\xi') ~|~ (x,\xi),(x,\xi') \in T^*_H S^2 \setminus 0 ; \xi|_{T_xH} = \xi'|_{T_xH} \}$$

Note that since $\partial_\theta$ is tangent to $H$, $(x,\xi)|_{TH} = (x,\xi')|_{TH}$ is equivalent to $I_1(x,\xi) = I_1(x,\xi')$. In order to get rid of the last two components of wave front set, we insert microlocal cutoff operators as in \cite{TZ1}. In this setting we can take them to be functions of the action operators $\widehat{I}_j$. Let $\phi_\epsilon$ and $\psi_\epsilon$ be smooth cutoff functions on $\R{}$ such that

\begin{equation}
    \phi_\epsilon(x) = \begin{cases} 1 ~ \text{for}~  \abs{x} \leq \epsilon/2 \\
    0 ~\text{for}~ \abs{x} > \epsilon
    \end{cases}
\end{equation}

\begin{equation}
    \psi_\epsilon(x) = \begin{cases} 1 ~ \text{for}~  \abs{x} > 1 - \epsilon/2 \\
    0 ~\text{for}~ \abs{x} < 1 - \epsilon
    \end{cases}
\end{equation}

Then we set $\widehat{\chi}_\epsilon^n = \phi_\epsilon(\frac{\widehat{I}_1}{\widehat{I}_2})$ and $\widehat{\chi}_\epsilon^t = \psi_\epsilon(\frac{\widehat{I}_1}{\widehat{I}_2})$. Finally set $\widehat{\chi}_\epsilon = \widehat{\chi}^n_\epsilon + \widehat{\chi}^t_\epsilon$. Note that the operator
$(I - \widehat{\chi}_\epsilon)$ has no wave front set in a conic $\epsilon/2$ neighborhood of both $N^*H$ and $T^*H$. We now define

\begin{equation}
    (\gamma_H^*\gamma_H)_{\geq \epsilon} = (I - \hat{\chi}_{\epsilon/2})\gamma_H^*\gamma_H(I - \hat{\chi}_\epsilon)
\end{equation}

\begin{equation}
    (\gamma^*_H \gamma_H)_{\leq \epsilon} = \hat{\chi}_{\epsilon/2} \gamma_H^*\gamma_H \hat{\chi}_{\epsilon}
\end{equation}

\begin{Prop} \label{prop:restrictiondecomp}
We have the decomposition 

\begin{equation}
   \gamma_H^*\gamma_H =  (\gamma_H^*\gamma_H)_{\geq \epsilon} + (\gamma_H^*\gamma_H)_{\leq \epsilon} + K_\epsilon
\end{equation}

where $\langle K_\epsilon \varphi_{\lambda_j} , \varphi_{\lambda_j} \rangle_{L^2(S^2,dV_g)} = O_\epsilon(\lambda_j^{-\infty})$ and $\varphi_{\lambda_j}$ are any orthonormal basis of eigenfunctions of $-\Delta_g$.

\end{Prop}

For the proof of this, see section 9.1.1 in \cite{TZ1}. We also quote the following description of the cutoff restriction operator:

\begin{Prop}
For each $\epsilon > 0$, $(\gamma_H^*\gamma_H)_{\geq \epsilon}$ is a Fourier integral operator in the class $I^{\frac{1}{2}}(M, M ; C_H)$ where $C_H$ is the homogeneous canonical relation 

\begin{equation}
    C_H = \{(x,\xi,x,\xi') \in T^*_H S^2 \setminus 0 \times T_H^*S^2 \setminus 0 ~|~ I_1(x,\xi) = I_1(x,\xi') \}
\end{equation}

In polar coordinates $(r,\theta,\rho,\eta)$ on $T^*S^2$, the set $C_H$ is parametrized by the map 

$$\iota_{C_H} : (\theta,\eta,\rho,\rho') \mapsto (r_0,\theta,\rho,\eta,r_0,\theta,\rho',\eta)$$

The half density part of the symbol of $(\gamma_H^*\gamma_H)_{\geq \epsilon}$ pulls back under $\iota_{C_H}$ to the half density 

\begin{equation}
    (1 - \chi_{\epsilon/2})(r_0,\theta,\rho,\eta)(1-\chi_{\epsilon})(r_0,\theta,\rho',\eta) \abs{d\theta \wedge d\eta \wedge d\rho \wedge d\rho'}^\frac{1}{2}
\end{equation}

\end{Prop}

This follows from Lemma 18 in \cite{TZ1} setting $\text{Op}_H(a) = \text{Id}$, because the geodesic polar coordinates $(r,\theta)$ are Fermi normal coordinates along $H$.

\subsection{The $I_2$ reflection map and the set $\widehat{C}_H$}

Here we include more geometric preliminaries to the description of the averaged restriction operator 

\begin{equation}
    \bar{V}_\epsilon = (2\pi)^{-2} \int_{T^2} U^*(\mathbf{t})(\gamma_H^*\gamma_H)_{\geq \epsilon}U(\mathbf{t}) \, d\mathbf{t}
\end{equation}

which is found in the next subsection. We begin by describing the so-called $I_2$ reflection map along $H$.

\begin{Prop} \label{prop:I2reflec}
Suppose $(x,\xi) \in T^*_H S^2$. If $(x,\xi) \notin T^*H$, there are is exactly one covector $(x,\xi') \in T_H^* S^2$ such that $I_2(x,\xi) = I_2(x,\xi')$, $(x,\xi) \neq (x,\xi')$ and $\xi|_{TH} = \xi'|_{TH}$. We refer to the map  

$$r_H : (x,\xi) \mapsto (x,\xi')$$

As the $I_2$-reflection map.
\end{Prop}

\begin{proof}

We'll show that on the set $\{ I_1 = c \}$, $I_2$ is an invertible function of the length $q(x,\xi) = \abs{\xi}^2_{g(x)}$. Thus, if $I_2(x,\xi) = I_2(x,\xi')$ and $I_1(x,\xi) = I_1(x,\xi')$, then $\abs{\xi}_{g(x)} = \abs{\xi'}_{g(x)}$ and this means that $(x,\xi') = (r_0,\theta,\pm \sqrt{ \abs{\xi}_{g(x)}^2 - c^2},c)$ in polar coordinates. The reflection map then flips the sign of the component dual to $r$. From \cite{CdV}, we have the formula

\begin{equation}
    I_2(x,\xi) = \int_{r_1}^{r_2} \sqrt{\abs{\xi}_{g(x)}^2 - \frac{p_\theta(x,\xi)^2}{a(r)^2}} \, dr + p_\theta
\end{equation}

Where $r_2$ and $r_1$ are the two solutions of $a(r) = \frac{p_\theta(x,\xi)}{\abs{\xi}_g}$. Now $r_1 = r_2$ if and only if $(x,\xi) \in T^*H$ thus we have that $r_1 \neq r_2$ and

$$\frac{\partial}{\partial\abs{\xi}} I_2(x,\xi) = \int_{r_1}^{r_2} \frac{\abs{\xi}_g}{\sqrt{\abs{\xi}^2_{g(x)} - \frac{c^2}{a(r)^2}}} \, dr > 0$$

This shows that $I_2$ is an increasing function of $\abs{\xi}_g$ on $\{I_1 = c\} \subset T^*S^2 \setminus 0$.

\end{proof}

From section 4.1, we know that for each $\epsilon > 0$, the operator $(\gamma_H^*\gamma_H)_{\geq \epsilon}$ is a Fourier integral operator with canonical relation 

$$C_H = \{(x,\xi,x,\xi') ~|~ (x,\xi),(x,\xi') \in T^*_HS^2 ; \xi|_{TH} = \xi'|_{TH} \}$$

In the study of $\bar{V}_\epsilon$, a related set appears. Define

\begin{equation}
    \widehat{C}_H = \{(x,\xi,x,\xi') ~|~ x \in H ; I_1(x,\xi) = I_1(x,\xi') ; I_2(x,\xi) = I_2(x,\xi') \}
\end{equation}

It is clear from proposition \ref{prop:I2reflec}, $\widehat{C}_H$ has the following simple description

\begin{Prop} \label{prop:Chatstructure}
The set $\widehat{C}_H$ is an immersed submanifold of dimension 3 which can be written as the union of the two embedded submanifolds

$$\widehat{C}_H = \Delta_{T^*_H S^2} \bigcup \text{graph} \, r_H|_{T^*_H S^2} $$

These intersect along the set $\Delta_{T^*H}$ where $\widehat{C}_H$ fails to be embedded.

\end{Prop}

\subsection{Description of the averaged restriction operator $\bar{V}_\epsilon$}

The purpose of this section is to describe the averaged restriction operator 

\begin{equation}
    \bar{V}_\epsilon = (2\pi)^{-2} \int_{T^2} U(\mathbf{t})^*(\gamma_H^*\gamma_H)_{\geq \epsilon}U(\mathbf{t}) \, d\mathbf{t}
\end{equation}

As a Fourier integral operator and calculate its symbolic data. In order to state the proposition, we set some notation. For any set $U \subset T^*S^2 \times T^*S^2$, we define its flow-out $\text{Fl}(U)$ by

$$\text{Fl}(U) = \bigcup_{\mathbf{t} \in T^2} \Phi_\mathbf{t} \times \Phi_\mathbf{t} (U) = \{(\Phi_\mathbf{t}(x,\xi),\Phi_\mathbf{t}(y,\eta)) ~|~ (x,\xi,y,\eta) \in U \}$$

In the calculation of the symbol of $\bar{V}_\epsilon$, there are two important submersions. Define $i_{D}, i_R : T^2 \times T^*_HS^2 \to T^*S^2 \times T^*S^2$ by 

\begin{equation}
    i_D(\mathbf{t},x,\xi) = (\Phi_\mathbf{t}(x,\xi), \Phi_\mathbf{t}(x,\xi))
\end{equation}

\begin{equation}
    i_R(\mathbf{t},x,\xi) = (\Phi_\mathbf{t}(x,\xi) , \Phi_\mathbf{t}(r_H(x,\xi)))
\end{equation}

The image of these maps are the diagonal and reflection flow-outs, $\text{Fl}(\Delta_{T^*_HS^2})$, $\text{Fl}(\text{graph}\, r_H|_{T^*_HS^2})$

\begin{Prop} \label{prop:flowoutsubmersions}
Both maps $i_D$ and $i_R$ are smooth submersions. Over any point $(y,\eta,y',\eta') \in T^*S^2 \times T^*S^2$ in the image of either map, the fiber can be identified with the set

\begin{equation} \label{set:fiber}
\{(x,\xi) \in T^*_HS^2 ~|~ \mathcal{P}(x,\xi) = \mathcal{P}(y,\eta) \}
\end{equation}

For $(y,\eta) \notin T^*_HS^2$, the fiber is identified with two distinct copies of $H$ corresponding to the choice of the northern or southern pointing covector lying on the torus $\mathcal{P}(y,\eta)$.

\end{Prop}

\begin{proof}

Fix a point $(y,\eta,y,\eta)$ in the image of $i_D$. Then $\Phi_{\mathbf{t}}(x,\xi) = (y,\eta)$ for some $\mathbf{t} \in T^2$ and $(x,\xi) \in T^*_HS^2$. The covector $(x,\xi)$ lies on the level set $\mathcal{P}^{-1}(y,\eta)$ and by proposition \ref{prop:I2reflec} there are two covectors in this set lying over $x$. Since the flow of $H_{I_1}$ translates around the equator, for each covector $(x,\xi)$ in the set \eqref{set:fiber}, there is a unique time $\mathbf{t}$ so that $\Phi_\mathbf{t}(x,\xi) = (y,\eta)$. In this way the fiber is identified with two copies of $H$

\end{proof}

These maps induce half densities on the flow-outs $\text{Fl}(\Delta_{T^*_HS^2})$ and $\text{Fl}(\text{graph}\, r_H|_{T^*_HS^2})$ as follows. We let $\mu^\frac{1}{2}$ be the half density on $T^2 \times T^*_HS^2$ which is equal to 1 on the product basis $\partial_\mathbf{t} \otimes \{\partial_\theta,\partial_\rho,\partial_\eta \}$. Then the exact sequence

$$0 \to \ker di_R \to T(T^2 \times T^*_HS^2) \to T(\text{Fl}(\text{graph}\, r_H|_{T^*_HS^2})) \to 0$$

implies that $\mu^\frac{1}{2} = \abs{d\theta}^\frac{1}{2} \otimes \mu^\frac{1}{2}/\abs{d\theta}^\frac{1}{2}$, where, under the identification of the fiber of $i$ with two copies of $H$, $\abs{d\theta}$ is the volume density such that $\int_H \abs{d\theta} = 2\pi$ and the quotient half density $\mu^\frac{1}{2}/\abs{d\theta}^{\frac{1}{2}}$ assigns the value $1$ to the basis $(d\Phi_\mathbf{t}v_i , d\Phi_\mathbf{t} dr_H v_i)$ where $v_i \in \{H_{I_2},\partial_\theta,\partial_\rho,\partial_\eta \}$. The same is true for the flowout of the diagonal replacing $i_R$ with $i_D$. In this case the quotient density $\mu^\frac{1}{2}$ assigns $1$ to the basis $(d\Phi_\mathbf{t}v_i,d\Phi_\mathbf{t}v_i)$.

\begin{Prop} \label{prop:Vbardesc}
The operator 

$$\bar{V}_\epsilon = (2\pi)^{-2} \int_{T^2} U^*(\mathbf{t}) (\gamma_H^*\gamma_H)_{\geq \epsilon} U(\mathbf{t}) \, d\mathbf{t}$$

is a Fourier integral operator in the class $I^0(S^2 \times S^2 ; C_{\bar{V}})$. Its canonical relation is 

$$C_{\bar{V}} = \text{Fl}(\widehat{C}_H) = \text{Fl}(\Delta_{T^*_HS^2}) \bigcup \text{Fl}(\text{graph} \, r_H|_{T^*_HS^2}) $$

The half density symbol of $\bar{V}_\epsilon$ is equal to 

$$\sigma(\bar{V}_\epsilon)(\Phi_\mathbf{t}(x,\xi),\Phi_\mathbf{t}(x,\xi')) = \frac{1}{\pi} (1 - \chi_\epsilon)(x,\xi) \left( \frac{\omega_2(x,\xi)}{\sqrt{1 - \frac{I_1^2(x,\xi)}{\abs{\xi}_g^2a(r_0)^2}}}\right)^\frac{1}{2} \frac{\mu^\frac{1}{2}}{\abs{d\theta}^\frac{1}{2}}$$

where $\mu^\frac{1}{2}/\abs{d\theta}^\frac{1}{2}$ is the half density induced by the fibrations of proposition \ref{prop:flowoutsubmersions}.

\end{Prop}

In order to analyze $\bar{V}_\epsilon$, we will view it as a composition of pullbacks and pushforwards applied to the Fourier integral operator

\begin{equation}
    V_\epsilon(\mathbf{t},\mathbf{t'}) = U(\mathbf{t})^*(\gamma_H^*\gamma_H)_{\geq \epsilon}U(\mathbf{t'})
\end{equation}

We begin by describing this operator. 

\begin{Prop} The operator $V_\epsilon(\mathbf{t},\mathbf{t'})$ is a Fourier integral operator in the class $I^{-\frac{1}{2}}(T^2 \times T^2 \times S^2 , S^2 ; C_V)$

\begin{equation}
    C_V = \{(\mathbf{t}, \mathcal{P}(x,\xi) , \mathbf{t}' , \mathcal{P}(x,\xi') , \Phi_{\mathbf{t}}(x,\xi), \Phi_{\mathbf{t}'}(x,\xi') ~|~ (x,\xi,x,\xi') \in C_H \} 
\end{equation}

\vspace{5mm}

The map $\iota_V : T^2 \times T^2 \times C_H \to T^*(T^2 \times T^2 \times S^2 \times S^2)$ given by

$$\iota_V : (\mathbf{t},\mathbf{t}',x,\xi,x,\xi') = (\mathbf{t}, \mathcal{P}(x,\xi) , \mathbf{t}' , \mathcal{P}(x,\xi') , \Phi_{\mathbf{t}}(x,\xi), \Phi_{\mathbf{t}'}(x,\xi'))$$

is a Lagrangian embedding whose image is $C_V$. The half density part of the principal symbol pulls back along $\iota$ to

$$\abs{d\mathbf{t} \wedge d\mathbf{t}'}^{\frac{1}{2}} \otimes \sigma((\gamma_H^*\gamma_H)_{\geq \epsilon})$$

\end{Prop}

\begin{proof}
Viewing both $U^*(\mathbf{t})$, $U(\mathbf{t'})$ as operators $U, U^* : C^\infty(S^2) \to C^\infty(T^2 \times S^2)$ then the composition we are talking about is really 

$$V_\epsilon(\mathbf{t},\mathbf{t'}) = Id \otimes U^*(\mathbf{t}) \circ Id \otimes (\gamma_H^*\gamma_H)_{\geq \epsilon} \circ U(\mathbf{t'})$$

The compositions are all transverse provided that $C_H$ and $C_U$ intersect transversely in the sense that the maps $\pi_i : C_H \to T^*S^2$ are transverse to the projections $\rho_i : C_U \to T^*S^2$ onto either factor. This follows from the fact that $C_U$ is essentially a canonical graph. It implies the orders add to give the stated order and one can check easily that the composite canonical relation and symbol is what was stated in the proposition.
\end{proof}

Now we describe the pullback under the time diagonal map. Let $\Delta : T^2 \times S^2 \times S^2 \to T^2 \times T^2 \times S^2 \times S^2$ be the map $\Delta : (\mathbf{t},x,y) \mapsto (\mathbf{t},\mathbf{t},x,y)$.

\begin{Prop}
The kernel of the operator $V_\epsilon(\mathbf{t}) = U^*(\mathbf{t})(\gamma_H^*\gamma_H)_{\geq \epsilon}U(\mathbf{t})$ is in the class $I^{-1}(T^2 \times S^2 \times S^2 ; \Delta^*C_V)$ Where $\Delta^*C_V$ is the pullback of $C_V$, the image of the Lagrangian embedding $i_{\Delta^*C_V} : T^2 \times C_H \to T^*(T^2 \times S^2 \times S^2)$ given by 

\begin{equation}
    \iota_{\Delta^*C_V} : (\mathbf{t},x,\xi,x,\xi) \mapsto (\mathbf{t}, \mathcal{P}(x,\xi) - \mathcal{P}(x,\xi'), \Phi_\mathbf{t}(x,\xi), \Phi_\mathbf{t}(x,\xi'))
\end{equation}

The half density symbol of $V_\epsilon(\mathbf{t})$ pulls back under $\iota_{\Delta^*C_V}$ to $\abs{d\mathbf{t}}^\frac{1}{2} \otimes \sigma((\gamma_H^*\gamma_H)_{\geq \epsilon}$.

\end{Prop}

\begin{proof}
Recall that the pullback of Lagrangian distributions is well-defined under a transversality condition. Namely, $V_\epsilon (\mathbf{t}) = \Delta^*V(\mathbf{t},\mathbf{t}')$ is a Lagrangian distribution as long as the maps $\pi|_{C_V} \to T^2 \times T^2 \times S^2 \times S^2$ and $\Delta$ are transverse, which is easily verified. Letting $N^*\Delta \subset T^*(T^2 \times S^2 \times S^2) \times T^*(T^2 \times T^2 \times S^2 \times S^2)$ be the co-normal bundle to the graph of $\Delta$ and $\pi : N^*\Delta \to T^*(T^2 \times T^2 \times S^2 \times S^2)$, projection onto the factor on the right, this implies that the pullback diagram 

$$\begin{tikzcd}
& F \arrow{r}{} \arrow{d}[swap]{}
& C_V \arrow{d}{\iota} \\
& N^*\Delta \arrow[swap]{r}{\pi}
& T^*(T^2 \times T^2 \times S^2 \times S^2 ) 
\end{tikzcd}$$

is transverse. The left projection of $F$ into $T^*(T^2 \times S^2 \times S^2)$ is then the set 

\begin{equation}
    \Delta^*C_V = \{\mathbf{t}, \mathcal{P}(x,\xi) - \mathcal{P}(x,\xi'), \Phi_\mathbf{t}(x,\xi), \Phi_\mathbf{t}(x,\xi') \}
\end{equation}

Which inherits a canonical half density determined by the symbol of $V_\epsilon(\mathbf{t},\mathbf{t}')$ on $C_V$, the canonical half density on $N^*\Delta \cong T^2 \times T^*S^2 \times T^*S^2$ and the symplectic half density on $T^*(T^2 \times T^2 \times S^2 \times S^2)$. This is the symbol of $V_\epsilon(\mathbf{t})$.
\end{proof}

Next, let $\pi : T^2 \times S^2 \times S^2 \to S^2 \times S^2$ be the projection onto the rightmost factors, $\pi(\mathbf{t},x,y) = (x,y)$. Let 
let $\pi_* : C^\infty(T^2 \times S^2 \times S^2) \to C^\infty (S^2 \times S^2)$ be the pushforward map defined on smooth functions by

$$\pi_*u(\mathbf{t},x,y) = (2\pi)^{-2} \int_{T^2} u(\mathbf{t},x,y) \, d\mathbf{t}$$

\begin{Lemma}
Let $N^*_\pi \subset T^*(T^2 \times S^2 \times S^2) \times T^*(S^2 \times S^2)$ denote the co-normal bundle to the graph of $\pi$ and $\rho_L : N^*_\pi \to T^*(T^2 \times S^2 \times S^2)$ denote the left projection. The pushforward diagram 

$$
\begin{tikzcd}
& F \arrow{r}{} \arrow{d}[swap]{}
& \Delta^*C_V \arrow{d}{\iota} \\
& N^*\pi \arrow[swap]{r}{\rho_L}
& T^*(T^2 \times S^2 \times S^2 ) 
\end{tikzcd}
$$

is clean away from the singular set $i_{\Delta^*C_V}(T^2 \times T^*H) \subset \Delta^*C_V$. 

\end{Lemma}

\begin{proof}
Recall that above diagram is clean if the fiber product $F$ is a submanifold of $\Delta^*C_V \times N^*\pi$ and the linearization

$$
\begin{tikzcd}
& TF \arrow{r}{} \arrow{d}[swap]{}
& T(\Delta^*C_V) \arrow{d}{d\iota} \\
& T(N^*\pi) \arrow[swap]{r}{d\rho_L}
& T(T^*(T^2 \times S^2 \times S^2 )) 
\end{tikzcd}
$$

is also a fiber product. Note that the fiber $F$ is the set 

$$F = \{(\mathbf{t},0,\Phi_\mathbf{t}(x,\xi),\Phi_\mathbf{t}(x,\xi'), \mathbf{t},0,\Phi_\mathbf{t}(x,\xi),\Phi_\mathbf{t}(x,\xi'), \Phi_\mathbf{t}(x,\xi) , \Phi_\mathbf{t}(x,\xi') ~|~ (x,\xi,x,\xi') \in \widehat{C}_H \}$$

The natural parametrization $i_F : T^2 \times \widehat{C}_H \to F$ is an embedding on the smooth parts of $\widehat{C}_H$. The image $i_F(T^2 \times T^*H)$ of the non-smooth part corresponds to the singular set $i_{\Delta^*C_V}(T^2 \times T^*H)$. Hence we see that $F$ is a submanifold of dimension 5 away from this set. To prove that the diagram is clean, we have to verify that $TF$ is given by the kernel of the map $\tau : T(\Delta^*C_V \times N^*_\pi) \to T(T^*(T^2 \times S^2 \times S^2))$ given by $\tau(u,v,w) = v - u$. Suppose that $u = di_{\Delta^*C_V}(\alpha,v,v') \in d\rho_L T(N^*_\pi)$. Then we have $(v,v') \in C_H$ with $dPv - dPv' = 0$. But this implies that $(v,v') \in T(\widehat{C}_H)$ and the tangent vector $(u,u,w) \in \ker \tau \subset T(\Delta^*C_V \times N^*_\pi)$ is actually equal to $di_F(\alpha,v,v')$, i.e. it is tangent to $F$.  
 
\end{proof}

Now since the pushforward diagram is clean, the right projection $\rho_R : F \to T^*(S^2 \times S^2)$ is a smooth submersion whose image

$$\rho_R(F) = C_{\bar{V}} = \text{Fl}(\Delta_{T^*_H S^2}) \bigcup \text{Fl}(\text{graph} \, r_H|_{T^*_HS^2})$$

is a Lagrangian submanifold of $T^*S^2 \times T^*S^2$. We now describe how the half densities on $N^*_\pi$ and $\Delta^*C_V$ determine a half density on the image $\rho_R(F) = C_{\bar{V}}$. More precisely, at each point $p \in F$, the clean diagram determines an element $\mu \otimes \nu^\frac{1}{2} \in \abs{\ker d(\rho_R)_p} \otimes \abs{T_{\rho_R(p)} C_{\bar{V}}}^\frac{1}{2}$. The half density at the point $q \in C_{\bar{V}}$ is then given by integrating the density over the fiber of $\rho_R$ over q:

\begin{equation}
    \left( \int_{\rho_R^{-1}(q)} \mu \right) \nu^\frac{1}{2}
\end{equation}

First consider the sequence of maps 

$$0 \to T_p F \to T_{i_F(p)} (\Delta^*C_V \times N^*_\pi) \to \im \tau \to 0$$

Where $\tau$ is the map above. Because the diagram is clean, this sequence is exact. We suppose that $p = i_F(\mathbf{t}, x,\xi,x,\xi')$. We will make use of several different bases which we pause to notate here. First, let $\mathcal{B} = (H_{I_2},\partial_\theta, \partial_\rho, \partial_\eta) \in T(T^*S^2)$. We will write $di_{N^*_\pi}(\partial_{\mathbf{t}} \otimes \mathcal{B})$ denote the basis on $T(N^*_\pi)$ obtained by pushing forward the product basis on $T^2 \times T^*S^2 \times T^*S^2$ determined by $\partial_{\mathbf{t}}$ and $\mathcal{B}$. We also let $\mathcal{B}'$ denote the basis $(\partial_\theta , \partial_\theta), (\partial_\eta , \partial_\eta),(\partial_\rho, 0) , (0,\partial_\rho) \in TC_H$ and similarly, $di_{\Delta^*C_V}(\partial_{\mathbf{t}} \otimes \mathcal{B}')$ denote the basis on $T(\Delta^*C_V)$ obtained by pushing forward the product basis on $T^2 \times C_H$.

\vspace{5mm}

Now, since both smooth branches of $\widehat{C}_H$ are graphs over $T^*_HS^2$, we have a natural half density $\mu^\frac{1}{2} \in \abs{T(T^2 \times \widehat{C}_H)}^{\frac{1}{2}}$ which pulls back to $\abs{d\mathbf{t}}^\frac{1}{2} \otimes \abs{d\theta \wedge d\eta \wedge d\rho}^\frac{1}{2}$ on $T^2 \times T^*_HS^2$. We let $\mathscr{B}$ be a basis of $T_pF$ such that $\mu^{\frac{1}{2}}(\mathscr{B}) = 1$. We complete this to a basis of $T(\Delta^*C_V \times N^*_\pi)$ by adding the 10 vectors $\mathbf{0} \otimes di_{N^*_\pi}(\partial_\mathbf{t} \otimes \mathcal{B}) $ in addition to the vector $(0 , d\mathcal{P}\partial_\rho, 0 , d\Phi_\mathbf{t} \partial_\rho , \mathbf{0})$. We claim that the change of basis matrix between this completed basis and the product basis $di_{\Delta^*C_{V}}(\partial_\mathbf{t} \otimes \mathcal{B}') \otimes \mathbf{0}$, $\mathbf{0} \otimes di_{N^*_\pi}(\partial_\mathbf{t} \otimes \mathcal{B})$ has determinant equal to $\pm 1$.

\begin{Lemma} \label{lem:sympbasechange}
Let $\abs{\Omega}^{\frac{1}{2}}$ denote the symplectic half density on $T^*S^2$. Then $$\Omega^\frac{1}{2}(\mathcal{B}) = \bigg{|}\frac{\partial I_2}{\partial \rho}\bigg{|}^\frac{1}{2}$$
\end{Lemma}

\begin{proof}
Since $(r,\theta,\rho,\eta)$ are canonical coordinates if we write $H_{{I}_2}$ in terms of the basis $\partial_r , \partial_\theta,\partial_\rho,\partial_\eta$, the coefficient of $\partial_r$ is $\frac{\partial I_2}{\partial \rho}$. Hence the change of basis from this symplectic basis to $\mathcal{B}$ has determinant $\abs{\frac{\partial I_2}{\partial \rho}}$
\end{proof}

Now let $\sigma \in \abs{T(\Delta^*C_V \times N^*_\pi)}^\frac{1}{2}$ denote the tensor product of the natural half density on $N^*_\pi$ and the symbol of $V_\epsilon(\mathbf{t})$ on $\Delta^*C_V$. Then in light of the lemma, $\sigma$ on the completed basis above is equal to 

\begin{equation} \label{eq:sigmaoncompletedbasis}
(1 - \chi_\epsilon)(x,\xi)\bigg{|}\frac{\partial I_2}{\partial \rho}(x,\xi)\bigg{|}
\end{equation}

This means that the exact sequence, together with our reference half density $\mu^\frac{1}{2}$ determines the half density $\nu^\frac{1}{2}$ on $\im \tau$ which assigns the value \eqref{eq:sigmaoncompletedbasis} to the 11 vectors $di_{N^*_\pi}(\mathbf{t} \otimes \mathcal{B})$, $(0 , -d\mathcal{P}\partial_\rho, 0 , -d\Phi_\mathbf{t} \partial_\rho)$. We complete this to a basis of $T(T^*(T^2 \times S^2 \times S^2))$ by adding the vector $(0, \partial_{\tau_1} , 0 , 0)$. Then the symplectic half density on this basis is equal to $\abs{\partial I_2 / \partial \rho}^\frac{3}{2}$. Hence, using the exact sequence

$$0 \to \im \tau \to T(T^*(T^2 \times S^2 \times S^2)) \to \text{coker} \, \tau \to 0$$

We get the negative half density on $\text{coker} \, \tau$ which assigns the value $(1 - \chi_\epsilon)(x,\xi)\abs{\partial I_2 / \partial \rho}^{-\frac{1}{2}}$ to the residue class of $(0,\partial_{\tau_1}, 0 , 0)$.To finish, we use the exact sequence associated the submersion $\rho_R$:

$$0 \to \ker d(\rho_R)_p \to T_p F \to T_{\rho_R(p)} C_V \to 0$$

Note that this is the exact sequence determined by either $i_D$ or $i_R$ of proposition \ref{prop:flowoutsubmersions} depending on whether $(x,\xi,x,\xi')$ is the diagonal or reflection branch of $\widehat{C}_H$. Now $\text{coker}\, \tau$ is symplectic dual to $\ker d\rho_R$. This allows us to identify the minus half density on $\text{coker}\, \tau$ with the half density 

$$(1 - \chi_\epsilon)(x,\xi) \bigg{|}\frac{\partial I_2}{\partial \rho}\bigg{|}^{-\frac{1}{2}} \abs{d\theta}^\frac{1}{2}$$

The symbol of $\bar{V}_\epsilon$ on the diagonal branch is therefore equal to

$$
\sigma(\bar{V}_\epsilon)(\Phi_\mathbf{t}(x,\xi),\Phi_{\mathbf{t}}(x,\xi)) = (2\pi)^{-2} \left( \int_{i_D^{-1}(\Phi_\mathbf{t}(x,\xi),\Phi_\mathbf{t}(x,\xi))} (1 - \chi_\epsilon)(y,\eta) \bigg{|}\frac{\partial I_2}{\partial \rho}(y,\eta) \bigg{|}^{-\frac{1}{2}} \abs{d\theta} \right) \frac{\mu^\frac{1}{2}}{\abs{d\theta}^\frac{1}{2}}
$$

and on the reflection branch we have 

$$
\sigma(\bar{V}_\epsilon)(\Phi_\mathbf{t}(x,\xi),\Phi_{\mathbf{t}}(r_H(x,\xi))) = (2\pi)^{-2} \left( \int_{i_R^{-1}(\Phi_\mathbf{t}(x,\xi),\Phi_\mathbf{t}(x,\xi))} (1 - \chi_\epsilon)(y,\eta) \bigg{|}\frac{\partial I_2}{\partial \rho}(y,\eta) \bigg{|}^{-\frac{1}{2}} \abs{d\theta} \right) \frac{\mu^\frac{1}{2}}{\abs{d\theta}^\frac{1}{2}}
$$

The proof is then completed by the following proposition:

\begin{Prop}
For $(x,\xi) \in T^*_HS^2$ in the support of the cutoff $1 - \chi_\epsilon(x,\xi)$, we have

\begin{equation} \label{eq:I2deriv}
    \frac{\partial I_2}{\partial \rho}(x,\xi) = \frac{\sqrt{1 - \frac{I_1^2(x,\xi)}{\abs{\xi}_g^2a(r_0)^2}}}{\omega_2(x,\xi)}
\end{equation}

where $\omega_2$ is the second component of the frequency vector $\omega_2 = \frac{\partial K}{\partial I_2}$.

\end{Prop}

\begin{proof}

We have $I_2 = G(\abs{\xi}_g , p_\theta)$. Since $p_\theta$ does not depend on $\rho$, 

$$\frac{\partial I_2}{\partial \rho} = \frac{\partial I_2}{\partial \abs{\xi}_g} \frac{\partial \abs{\xi}_g}{\partial \rho}$$

Now for $(x,\xi) \in T^*_HS^2$, we have $\abs{\xi}_g = \sqrt{\rho^2 + \frac{p_\theta^2}{a(r_0)^2}}$. So
$\frac{\partial I_2}{\partial \abs{\xi}_g} = \omega_2^{-1}(x,\xi)$ and

$$\frac{\partial \abs{\xi}_g}{\partial \rho} = \frac{\sqrt{\abs{\xi}_g^2 - \frac{p_\theta^2}{a(r_0)^2}}}{\abs{\xi}_g}$$

\end{proof}

Since the symbol of the cutoff, $\chi_\epsilon$ and all of the quanities appearing in \eqref{eq:I2deriv} are functions of $I_1$ and $I_2$, they are constant on the fibers of $i_D$ and $i_R$. Hence the integrals appearing above can be simplified to 

$$
\sigma(\bar{V}_\epsilon)(\Phi_\mathbf{t}(x,\xi),\Phi_{\mathbf{t}}(x,\xi)) = \frac{1}{\pi} (1 - \chi_\epsilon)(x,\xi) \left( \frac{\omega_2(x,\xi)}{\sqrt{1 - \frac{I_1^2(x,\xi)}{\abs{\xi}_g^2a(r_0)^2}}} \right)^\frac{1}{2} \frac{\mu^\frac{1}{2}}{\abs{d\theta}^\frac{1}{2}}
$$

$$
\sigma(\bar{V}_\epsilon)(\Phi_\mathbf{t}(x,\xi),\Phi_{\mathbf{t}}(r_H(x,\xi))) = \frac{1}{\pi} (1 - \chi_\epsilon)(x,\xi) \left( \frac{\omega_2(x,\xi)}{\sqrt{1 - \frac{I_1^2(x,\xi)}{\abs{\xi}_g^2a(r_0)^2}}} \right)^\frac{1}{2} \frac{\mu^\frac{1}{2}}{\abs{d\theta}^\frac{1}{2}}
$$

This completes the proof of proposition \ref{prop:Vbardesc}. We now want to show that $\bar{V}_\epsilon$ can be written as the sum of a pseudo-differential operator and a Fourier integral operator. 

\begin{Prop} \label{prop:symbolofPandF}
We have a decomposition $\bar{V}_\epsilon = P_\epsilon + F_\epsilon$ where $P_\epsilon$ is an order zero pseudo-differential operator with scalar symbol equal to 

$$\sigma(P_\epsilon)(y,\eta) = \frac{1}{\pi} (1-\chi_\epsilon)(y,\eta) \frac{\omega_2(y,\eta)}{\sqrt{1 - \frac{p_\theta^2(y,\eta)}{\abs{\eta}_y^2 a(r_0)^2}}} \abs{dy \wedge d\eta}^\frac{1}{2}$$

 $F_\epsilon \in I^0(S^2 \times S^2; \text{Fl}(\text{graph}\, r_H|_{T^*_HS^2}))$. The symbol of $F_\epsilon$ is the half density 

$$
\sigma(F_\epsilon)(\Phi_\mathbf{t}(x,\xi), \Phi_\mathbf{t}(r_H(x,\xi))) = \frac{1}{\pi} (1 - \chi_\epsilon)(x,\xi) \left( \frac{\omega_2(x,\xi)}{\sqrt{1 - \frac{I_1^2(x,\xi)}{\abs{\xi}_g^2 a(r_0)^2}}}\right)^\frac{1}{2} \frac{\mu^\frac{1}{2}}{\abs{d\theta}^\frac{1}{2}}
$$

where $\mu^\frac{1}{2}/\abs{d\theta}^\frac{1}{2}$ is the half density on the flow-out of the reflection graph determined in proposition \ref{prop:flowoutsubmersions}.

\end{Prop}

\begin{proof}
Note that the two flow-out sets $\text{Fl}(\Delta_{T^*_HS^2}) \bigcup \text{Fl}(\text{graph}\,r_H|_{T^*_HS^2}$ are disjoint when $(x,\xi)$ is restricted to the support of a the cutoff $1 - \chi_\epsilon$. Since $V_\epsilon$ only has wave front set in the flow-outs of this region, we can let $\Psi \in C^\infty_c(T^*S^2 \times T^*S^2)$ be a smooth cutoff function such that $\psi = 1$ in a neighborhood of the diagonal flow-out and has support disjoint from the reflection flow-out. Then we have

$$\bar{V}_\epsilon = \widehat{\psi}\bar{V}_\epsilon + (I - \widehat{\psi})\bar{V}_\epsilon$$ 

The diagonal flow-out is inside $\Delta_{T^*S^2}$ so the first term is a pseudo-differential operator. The symbol is unchanged due to the fact that $\psi$ and $1 - \psi$ are equal to 1 on neighborhoods of the diagonal, reflected flow-outs. On the diagonal branch of the flow-out, we also have the natural symplectic half density $\abs{dy \wedge d\eta \wedge dy \wedge d\eta}^\frac{1}{2}$. It is easy to check that (see lemma \ref{lem:sympbasechange})

$$\frac{\mu^\frac{1}{2}}{\abs{d\theta}^\frac{1}{2}} = \bigg{|}\frac{\partial I_2}{\partial \rho}\bigg{|}^{-\frac{1}{2}} \abs{dy \wedge d\eta \wedge dy \wedge d\eta}^\frac{1}{2}$$

This accounts for the difference between the symbol of $P_\epsilon$ stated here and the symbol of $\bar{V}_\epsilon$ on the diagonal branch.

\end{proof}

\section{Calculation of weak-* limits: Proof of theorem 1.1}

In this section we compute the weak-* limits of the measures \eqref{eq:defnnuell}, \eqref{eq:defnmuell} by expanding their un-normalized versions in $\ell$. Recall that we let $\Pi_\ell : L^2(S^2,dV_g) \to L^2(S^2,dV_g)$ denote the orthogonal projection onto the $\widehat{I}_2 = \ell$ eigenspace. And if $A : C^\infty(S^2) \to C^\infty(S^2)$ is an operator which commutes with $D_\theta$. Then we have 

\begin{equation} \label{eq:traceformula}
    \text{Trace} \, f \left( \frac{D_\theta}{\ell} \right) A\Pi_\ell = \sum_{m = -\ell}^\ell \langle A\varphi^\ell_m , \varphi^\ell_m \rangle f\left( \frac{m}{\ell}\right)
\end{equation}

We will use the symbol calculus to expand the left hand side of \eqref{eq:traceformula} in powers of $\ell$. To begin with, we need a description of the operator $f(D_\theta / \ell)$.

\begin{Prop} \label{prop:angularpdo}
Let $f \in C^\infty_c(\R{})$. The operator $f\left( \frac{D_\theta}{\ell} \right)$ is a semi-classical pseudo-differential operator in the class $\Psi^{-\infty}_{\ell^{-1}}(S^2)$ with principal symbol equal to $f(p_\theta(y,\eta))$.

\end{Prop}

\begin{proof}

Note that by Fourier inversion, we can write 

\begin{equation} \label{eq:fourierinv}
    f\left( \frac{D_\theta}{\ell} \right) = \frac{1}{2\pi} \int_{\R} \widehat{f}(t) e^{i\frac{t}{\ell}D_\theta} \, dt
\end{equation}

Becauase the flow of $D_\theta$ is just linear translation in the polar coordinates $(r,\theta,\rho,\eta)$, we can write

$$(\exp i\frac{t}{\ell}D_\theta)(r,\theta,r',\theta') = (2\pi)^{-2}\int_{\R^2} e^{i[(r-r')\rho + (\theta - \theta')\eta]} e^{i\frac{t}{\ell}\eta} \, d\rho \, d\eta$$

Now change variables $\rho' = \rho/\ell$, $\eta = \eta/\ell$. Then

$$(\exp i\frac{t}{\ell}D_\theta)(r,\theta,r',\theta') = \frac{\ell^2}{(2\pi)^2}\int_{\R^2} e^{i\ell[(r-r')\rho + (\theta - \theta')\eta]} e^{it\eta'} \, d\rho' \, d\eta'$$

Inserting this expression into \eqref{eq:fourierinv} and integrating in $t$ finishes the proof.

\end{proof}

We also need a description of $\Pi_\ell$ as a semi-classical Fourier integral operator. For details, see for instance theorem 1 of \cite{Z1}. Although this is written for the cluster projection of a Zoll Laplacian, the same argument applies to the operator $\widehat{I}_2$ considered here.

\begin{Prop} \label{prop:piell}
For $A \in \Psi^0$ a homogeneous order zero pseudo-differential operator, $A\Pi_\ell$ is a semi-classical Fourier integral operator of order $\frac{1}{2}$ associated to the canonical relation

$$C_{\Pi} = \{(x,\xi,y,\eta) \in \Sigma \times \Sigma ~|~ \exists t \in [0,2\pi) \exp tH_{I_2}(x,\xi) = (y,\eta) \}$$

Where $\Sigma = \{I_2 = 1 \}$. Along the parametrizing map $\iota_{\Pi} : S^1 \times \Sigma \to T^*S^2 \times T^*S^2$

$$\iota_{\Pi} : (t,x,\xi) \mapsto (x,\xi, \exp tH_{I_2}(x,\xi)) $$

The half density symbol pulls back to

$$\iota_{\Pi}^*\sigma(A\Pi_\ell) = \ell^{\frac{1}{2}}e^{-i \ell t} \abs{dt}^\frac{1}{2} \otimes \sigma(A) \abs{d\mu_L}^\frac{1}{2}$$

Where $d\mu_L$ is Liouville measure on the energy surface $\Sigma$ and $\sigma(A)$ is the scalar symbol of $A$ with respect to the canonical symplectic half density on $N^*\Delta$.

\end{Prop}

\subsection{Weak-* limit of $\nu_\ell(B)$}

Let $B \in \Psi^0(S^2)$ and $\bar{B}$ be the average \eqref{eq:torusavg} of $B$ with respect to the torus action $U(\mathbf{t})$. Then the un-normalized version of $\nu_\ell(B)$ tested against $f \in C_c^\infty(-1,1)$ is

$$\sum_{m = -\ell}^\ell \langle B \varphi^\ell_m , \varphi^\ell_m \rangle f\left( \frac{m}{\ell} \right) = \text{Trace}\, f\left( \frac{D_\theta}{\ell}\right)\bar{B} \Pi_\ell$$

The right hand side is the trace of a semi-classical Fourier integral operator and by standard symbol calculus it has the leading order asymptotics 

$$\sum_{m = -\ell}^\ell \langle B \varphi^\ell_m , \varphi^\ell_m \rangle f\left( \frac{m}{\ell} \right) = \ell\int_{\Sigma} f(p_\theta) \sigma(\bar{B}) \, d\mu_L + O(1)$$

Similarly, the normalizing coefficient $N_\ell$ is

$$N_\ell = \text{Trace} \, \bar{B}\Pi_\ell = \ell \int_{\Sigma} \sigma(\bar{B}) \, d\mu_L + O(1)$$

Finally, since $\sigma(\bar{B})$ is just the average of $\sigma(B)$ with respect to the torus action $\Phi_\mathbf{t}$, we have $\int_\Sigma \sigma(\bar{B}) \, d\mu_L = \int_\Sigma \sigma(B) \, d\mu_L = \omega(B)$. We also write 

$$\int_\Sigma f(p_\theta)\sigma(\bar{B}) \, d\mu_L = \int_{-1}^1 f(c) \int_{T_c} \sigma(\bar{B}) \, d\mu_{L,c} dc  = \int_{-1}^1 f(c) \widehat{\sigma(B)}(c) \, dc$$

This completes the proof of theorem 1.1 (b) when $f$ is compactly supported. As for $\mu_\ell$, the full statement follows from the fact that $\widehat{\sigma(B)}(c)$ is an $L^1$ function on $[-1,1]$.

\subsection{Weak-* limit of $\mu_\ell$}

To begin with, we need to relate the un-normalized version of $\eqref{eq:defnmuell}$ to a trace formula.

\begin{Prop}
Let $f \in C^\infty_c(-1,1)$. For each $\epsilon > 0$, 

\begin{equation}
    \sum_{m = -\ell}^\ell \lpnorm{\varphi^\ell_m}{2}{H}^2 f\left(\frac{m}{\ell}\right) = \text{Trace} \, f\left(\frac{D_\theta}{\ell} \right) \bar{V}_\epsilon \Pi_\ell + R(\epsilon,\ell)
\end{equation}

where 

$$\limsup_{\ell \to \infty} \frac{\abs{R(\epsilon,\ell)}}{\ell} = O(\epsilon) $$

\end{Prop}

\begin{proof}

Note that by proposition \ref{prop:restrictiondecomp}, we have 

\begin{equation}
\begin{aligned}
\sum_{m = -\ell}^{\ell} \lpnorm{\varphi^\ell_m}{2}{H}^2 f\left( \frac{m}{\ell} \right) &= \sum_{m = -\ell}^{\ell} \langle (\gamma_H^*\gamma_H)_{\geq \epsilon} \varphi^\ell_m , \varphi^\ell_m \rangle f\left( \frac{m}{\ell} \right) + \sum_{m = -\ell}^{\ell} \langle (\gamma_H^*\gamma_H)_{\leq \epsilon} \varphi^\ell_m , \varphi^\ell_m \rangle f\left( \frac{m}{\ell} \right) \\
&+ \sum_{m = -\ell}^\ell \langle K_\epsilon \varphi^\ell_m , \varphi^\ell_m \rangle f \left( \frac{m}{\ell} \right)
\end{aligned}
\end{equation}

The first term on the right hand side is just the trace appearing in the proposition. Further, since $\abs{\langle K_\epsilon \varphi^\ell_m , \varphi^\ell_m \rangle} = O_\epsilon (\ell^{-\infty})$, we just need to show that 

\begin{equation}
    \limsup \frac{1}{\ell} \bigg{|}\sum_{m = -\ell}^{\ell} \langle (\gamma_H^*\gamma_H)_{\leq \epsilon} \varphi^\ell_m , \varphi^\ell_m \rangle f\left( \frac{m}{\ell} \right) \bigg{|} = O(\epsilon) 
\end{equation}

As in the discussion on page 37 of \cite{TZ1}, we can bound the sum 

$$\frac{1}{\ell} \bigg{|}\sum_{m = -\ell}^{\ell} \langle (\gamma_H^*\gamma_H)_{\leq \epsilon} \varphi^\ell_m , \varphi^\ell_m \rangle f\left( \frac{m}{\ell} \right) \bigg{|}$$

By a sum of terms of the form 

$$\frac{1}{\ell} \sum_{m = -\ell}^\ell ||\gamma_H \widehat{\chi}_\epsilon^j \varphi^\ell_m||^2_{L^2}(H)$$

where $\widehat{\chi}^j_\epsilon$ is either the tangential or the normal cutoff operator. In both cases, the symbol of the operator appearing is supported inside a set of volume $O(\epsilon)$ inside $\Sigma$. By the pointwise Weyl law,

$$\limsup_{\ell \to \infty } \frac{1}{\ell} \sum_{-\ell}^\ell \abs{\widehat{\chi}^j_\epsilon \varphi^\ell_m(x)}^2 = O(\epsilon) $$

 and integrating this along $H$ preserves this bound.

\end{proof}

\begin{Prop}
For each $\epsilon > 0$,

$$\text{Trace} \, f\left(\frac{D_\theta}{\ell} \right) \bar{V}_\epsilon \Pi_\ell = 4\pi \ell \left(\int_{-1}^1 f(c)(1 - \chi_\epsilon)(c)  \frac{\omega_2(c,1)}{\sqrt{1 - \frac{c^2}{K(c,1)^2 a(r_0)^2}}} \, dc \right) + O_\epsilon(1)$$

\end{Prop}

\begin{proof}

By proposition \ref{prop:symbolofPandF}, we have $\bar{V}_\epsilon = P_\epsilon + F_\epsilon$. From propositions \ref{prop:angularpdo},\ref{prop:piell}, and \ref{prop:symbolofPandF}, the contribution of the $P_\epsilon$ term in the trace is equal to 

$$\ell\left( \int_\Sigma f(p_\theta)\sigma(P_\epsilon) \, d\mu_L \right) + O_\epsilon(1)$$

Since the symbol of $P_\epsilon$ is a function of $I_1$ and $I_2$, it is constant on each torus $T_c$ and the leading term is equal thus equal to 

$$(2\pi)^2 \ell \int_{-1}^1 f(c) \sigma(P_\epsilon)(c,1) \, dc$$

which is the stated term in the proposition. To finish the proof, we need to show that the contribution to the trace from the $F_\epsilon$ piece is of size $O_\epsilon(1)$. For this, note that $f\left(\frac{D_\theta}{\ell}\right) F_\epsilon \Pi_\ell$ is a semi-classical Fourier integral operator of order $\frac{1}{2}$ associated to the canonical relation 

$$C_{R\Pi} = \{(x,\xi,y,\eta) ~|~ (x,\xi) = \Phi_\mathbf{t}(r_H(x',\xi')) ~\text{and}~ (\Phi_\mathbf{t}(x',\xi'),y,\eta) \in C_\Pi \}$$

The trace is controlled by the symbol on the intersection $C_{R\Pi} \cap \Delta_{T^*S^2}$. This is equal to the set

$$\{ (\Phi_\mathbf{t}(x',\xi'), \Phi_\mathbf{t}(r_H(x',\xi')) \in C_\Pi ~|~ \mathbf{t} \in T^2, (x',\xi') \in T^*_HS^2 \} $$

And this is equivalent to the statement that $(x',\xi')$ and $r_H(x',\xi')$ lie along the same $I_2$ bicharacteristic. But if $(x',\xi') \notin T^*H$, this would mean that the projection of the $I_2$ bicharacteristic to $S^2$ has a self-intersection, which is impossible. Thus it must be that $(x',\xi') = r_H(x',\xi') \in T^*H$. Due to the cutoff $\chi_\epsilon$, the symbol of $F_\epsilon$ vanishes on the aforementioned set. Hence the order $\ell$ term in the trace vanishes as claimed.

\end{proof}

\begin{Prop}
The normalizing factor $M_\ell = \sum_{m = -\ell}^\ell \lpnorm{\varphi^\ell_m}{2}{H}^2$ satisfies

$$\lim_{\ell \to \infty} \frac{M_\ell}{\ell} = 4\pi \int_{-1}^1 \frac{\omega_2(c,1)}{\sqrt{1 - \frac{c^2}{K(c,1)^2 a(r_0)^2}}} \, dc $$
\end{Prop}

\begin{proof}

In the same fashion as the proof of proposition 5.3, we can write 

$$M_\ell = \text{Trace} \, \bar{V}_\epsilon \Pi_\ell + R'(\epsilon,\ell)$$

$$\text{Trace} \, \bar{V}_\epsilon \Pi_\ell = \ell \int_{-1}^1 (1 - \chi_\epsilon)(c)\frac{\omega_2(c,1)}{\sqrt{1 - \frac{c^2}{K(c,1)^2 a(r_0)^2}}} \, dc + O_\epsilon(1)$$

where $\limsup_{\ell \to \infty} \abs{R'(\epsilon,\ell)}/\ell = O(\epsilon) $. Since 

$$\int_{-1}^1 (1 - \chi_\epsilon)(c)\frac{\omega_2(c,1)}{\sqrt{1 - \frac{c^2}{K(c,1)^2 a(r_0)^2}}} \, dc \to \int_{-1}^1 \frac{\omega_2(c,1)}{\sqrt{1 - \frac{c^2}{K(c,1)^2 a(r_0)^2}}} \, dc$$ 

as $\epsilon \to 0$, the statement follows.

\end{proof}

Now in light of propositions 5.1,5.2,and 5.3, for $f \in C^\infty_c(-1,1)$, 

$$\langle \mu_\ell , f \rangle = \frac{1}{M_\ell}\sum_{m = - \ell}^{\ell} \lpnorm{\varphi^\ell_m}{2}{H}^2 f\left( \frac{m}{\ell} \right) = 4\pi \frac{\ell}{M_\ell} \left(\int_{-1}^1 f(c)(1 - \chi_\epsilon)(c)  \frac{\omega_2(c,1)}{\sqrt{1 - \frac{c^2}{K(c,1)^2 a(r_0)^2}}} \, dc \right) + R''(\epsilon,\ell)$$

where $\limsup \abs{R''(\epsilon,\ell)} = O(\epsilon)$. Taking $\ell \to \infty$ and then $\epsilon \to 0$ finishes the proof of theorem 1.1 (a) when $f$ is compactly supported. We can freely upgrade this statement to $f \in C^0([-1,1])$ because $$\frac{\omega_2(c,1)}{\sqrt{1 - \frac{c^2}{K(c,1)^2 a(r_0)^2}}}$$ is an $L^1$ function of $c$ on $[-1,1]$.

\section{Unitary conjugation to the round sphere}

In this section we prove theorem 1.2. That is, we construct a unitary Fourier integral operator $W : L^2(S^2,g) \to L^2(S^2,g_{can})$ such that

$$
\begin{cases}
W \hat{I}_2 W^* = A \\
W D_\theta W ^* = D_\theta
\end{cases}
$$

Where $A = \sqrt{-\Delta_{g_can} + \frac{1}{4}} - \frac{1}{2}$ is the degree operator on the round sphere. We begin by describing the outline of the proof. First, using the canonical transformation $\chi: T^*S^2 \setminus 0 \to T^*S^2 \setminus 0$ of section 2.2 which satisfies $\chi^*I_2 = \abs{\xi}_{g_can}$, $\chi^*p_\theta = p_\theta$, we can find a unitary Fourier integral operator $W_0$ so that $[W_0,D_\theta] = 0$ and 

\begin{equation}
W_0 \hat{I}_2 W_0^* = A + R_{-1}
\end{equation}

where $R_{-1}$ is a pseudo-differential operator of order $-1$. We then use the  averaging argument of Guillemin (See \cite{Gu}) to show that there exists a unitary pseudo-differential operator $F$ of order zero such that

\begin{equation}
F(A + R_{-1}) F^* = A + R^{\#}_{-1}
\end{equation}

where $[A , R^{\#}_{-1}] = 0$ and $[F , D_\theta] = 0$. This is contained in propositions \ref{prop:firstconj}, \ref{prop:secondconj}, and \ref{prop:thirdconj}. Then $W = FW_0$ is a unitary Fourier integral operator which commutes with $D_\theta$ and conjugates $\hat{I}_{2}$ to $A + R_{-1}^{\#}$, where $R_{-1}$ is an order $-1$ pseudo commuting with $A$. Using the fact that $A + R_{-1}^{\#}$ and $A$ have the same spectrum, we easily see that $R^{\#}$ is a finite rank operator.

\begin{Prop} \label{prop:firstconj}
There exists a unitary Fourier integral operator $W_0$ such that $W_0 \hat{I}_2 W_0^* = A + R_{-1}$ where $R_{-1} \in \Psi^{-1}$ is self-adjoint and $[W_0, D_\theta] = 0$. In this case we also have $[R_{-1}, D_\theta] = 0$
\end{Prop}

\begin{proof}

Let $U_0$ be any unitary Fourier integral operator whose canonical relation is the graph of $\chi$. Then by Egorov's theorem,

\begin{equation} \label{eq:egorov}
U_0\hat{I}_2U_0^* = A + R
\end{equation}

Where $R \in \Psi^0$. Both the left hand side and $A$ are self-adjoint, so $R$ is as well. The subprincipal symbols of both the left hand side and $A$ vanish which implies that $\sigma(R) = 0$ so $R \in \Psi^{-1}$. We write $R_{-1}$ from now on to emphasize this. The only thing left to do is to show that we can modify $U_0$ in order to make it commute with $D_\theta$. We let $V(t) = \exp itD_\theta$ and set 

\begin{equation}
W'_0 = \frac{1}{2\pi} \int_0^{2\pi} V(t)U_0V(-t) \, dt
\end{equation}

$W'_0$ is a Fourier integral operator with the same canonical relation as $U_0$, although it may not be unitary. To fix this, replace $W'_0$ with $W_0 = [W_0'(W_0')^*]^{-\frac{1}{2}}W'_0$. Then $W_0W_0^* = I$ and $W$ is still a Fourier integral operator associated to the same canonical graph since $W_0'(W'_0)^*$ is pseudo-differential. $W'_0$ commutes with $D_\theta$ so $W_0$ does as well. Note that if one replaces $U_0$ by $W_0$, \eqref{eq:egorov} is still valid since both operators are associated to the graph of $\chi$. Since $\widehat{I}_2$ and $A$ commute with $D_\theta$, we automatically have that $R_{-1}$ does as well.

\end{proof}

The following two propositions constitute a slight modifcation of what Guillemin refers to as the averaging lemma, found in \cite{Gu}. The goal of the modification is to make sure the conjugations commute with $D_\theta$.

\begin{Prop} \label{prop:secondconj}
Let $R_{-1}$ be as in proposition \ref{prop:firstconj}. Then there exists a unitary pseudo-differential operator $F \in \Psi^0$, a self-adjoint operator $R_{-1}^{\#} \in \Psi^{-1}$ which commutes with $A$ and a smoothing operator $R_{-\infty}$ such that $F(A + R_{-1})F^* = A + R_{-1}^{\#} + R_{-\infty}$  and $[F, D_\theta] = 0$
\end{Prop}

\begin{proof}

We let $U(t) = \exp (itA)$ be the unitary group generated by $A$ and for a pseudo-differential operator $B$, define as before, its average with respect to $U(t)$ by

\begin{equation}
B_{av} = \frac{1}{2\pi} \int_0^{2\pi} U(t)BU(-t) \, dt
\end{equation}

Then $B_{av}$ commutes with $A$ and is self-adjoint if $B$ is. We recall the statement of lemma 2.1 in \cite{Gu}: If $R$ is any self-adjoint operator of order $-k$, $k \in \mathbb{N}$, there exists a skew-adjoint pseudodifferential operator $S$ of order $-k$ so that $[A , S] = R - R_{av} + \Psi^{-k - 1}$. This statement is equivalent to the vanishing of the principal symbol of $[A,S] - (R - R_{av})$ which is a first order transport equation for $\sigma(S)$. This can be solved for $\sigma(S)$ explicitly on $S^*S^2$ , which can be extended as a degree $-k$ homogeneous function to $T^*S^2 \setminus 0$. Since it is imaginary, we can choose $S$ to be skew-adjoint. Given such an $S$, let $V(t) = \exp(itD_\theta)$ and set $\bar{S} = (2\pi)^{-1} \int_0^{2\pi} V(t)SV(-t) \, dt$. Then $\bar{S}$ is still skew-adjoint and commutes with $D_\theta$. If we further suppose that $R$ commutes with $D_\theta$ then

\begin{align}
[A,\bar{S}] =& \frac{1}{2\pi} \int_0^{2\pi} V(t)[A,S]V(-t) \, dt \\
=& \frac{1}{2\pi}\int_0^{2\pi} V(t)(R - R_{av})V(-t) \, dt + \Psi^{-k-1} \\
=& \, R - R_{av} + \Psi^{-k-1}
\end{align}

Hence we may assume from the outset that $[S, D_\theta] = 0$. This fact allows us to build the operator $F$ in stages. If $R_{-1}$ is the operator in proposition \ref{prop:firstconj}, then using the above procedure we can choose $S_{-1} \in \Psi^{-1}$ skew-adjoint such that 

\begin{equation}
[A , S_{-1}] = R_{-1} - (R_{-1})_{av} + R_{-2}
\end{equation}

where $R_{-2} \in \Psi^{-2}$ and so that $[S_{-1}, D_\theta] = 0$. Then setting $F_1 = \exp {S_{-1}}$, a direct calculation shows that

\begin{equation}
F_1(A + R_{-1})F_1^* = A + (R_{-1})_{av} + R_{-2}
\end{equation}

By construction, $F_1$ is unitary and commutes with $D_\theta$. We can now choose $S_{-2}$ skew-adjoint commuting with $D_\theta$ such that 

\begin{equation}
[A , S_{-2}] = R_{-2} - (R_{-2})_{av} + R_{-3}
\end{equation}

Then, with $F_2 = \exp {S_{-2}} \exp{S_{-1}}$ we have

\begin{equation}
F_2(A + R_{-1}) = A + (R_{-1})_{av} + (R_{-2})_{av} + R_{-3}
\end{equation}

Continuing in this way, we get a sequence of unitary operators $$F_k = \exp S_{-k} \cdots \exp S_{-1}$$ so that $F_k$ commutes with $D_\theta$ and 

\begin{equation}
F_k(A + R_{-1})F_k^* = A + (R_{-1})_{av} + \cdots + (R_{-k})_{av} + R_{-k - 1}
\end{equation}

We also note that $F_{k+1} - F_k \in \Psi^{-k-1}$. Let $F' \sim \sum_{k=1}^\infty (F_{k+1} - F_k)$, $R \sim \sum_{k=1}^\infty (R_{-k})_{av}$, and $R_{-1}^{\#} = R_{av}$. Then we know that $R_{-1}^{\#} - R \in \Psi^{-\infty}$ and if we put $F = F' + F_1$ we have $F - F_k \in \Psi^{-k}$. It is then easy to check that

\begin{equation}
F(A + R_{-1})F^* - (A + R_{-1}^{\#}) \in \Psi^{-\infty}
\end{equation}

Furthermore, since all of the $F_k$ commute with $D_\theta$, we can choose $F$ so that it does as well. As in the proof of proposition \ref{prop:firstconj}, $F$ may not be unitary. This is fixed in the same way, by replacing $F$ with $(FF^*)^{-\frac{1}{2}}F$. More explicitly, let $G = FF^* - I$. Note that $F = F_k + \Psi^{-k}$ which implies that $G$ is a smoothing operator. By the functional calculus, we can find a self-adjoint operator $K$ so that $(I + K)^2 = (I + G)^{-1}$ and if we replace $F$ by $(I + K)F$, then $F$ is unitary, $[F, D_\theta] = 0$, and we still have $F - F_k \in \Psi^{-k}$ since $K$ is a smoothing operator.

\end{proof}

\begin{Prop} \label{prop:thirdconj}
Suppose that $R^{\#}_{-1}$ and $R_{-\infty} \in \Psi^{-\infty}$ are as in proposition \ref{prop:secondconj} and that $\text{Spec}(A + R^{\#}_{-1} + R_{-\infty}) = \text{Spec}(A) = \mathbb{N}$. Then there exists a unitary operator $L$ and $R^{\#} \in \Psi^{-1}$, self-adjoint, such that $[R^{\#}, A] = 0$ and

\begin{equation}
L(I + R + R_{-\infty})L^* = I + R^{\#}  
\end{equation}

\vspace{3mm}

Furthermore, $L - I$ is a smoothing operator and $[L, D_\theta] = 0$

\end{Prop}

\begin{proof}
Let $V_k$ denote the $k^{th}$ eigenspace of $A$ and $V'_k$ the $k^{th}$ eigenspace of $A + R_{-1}^{\#} + R_{-\infty}$. Also let $\pi_k$ and $\pi'_k$ denote orthogonal projection onto these subspaces. Finally let $P_k = \pi'_k$ restricted to $V'_k$. First we show that there is a $C > 0$ so that for all $N \geq 0$ and $k$ sufficiently large

\begin{equation}
||(A + R_{-1}^{\#})^N (P_k - \pi'_k)||_{L^2}  \leq C ||(A + R_{-1}^{\#})^N R_{-\infty} \pi'_k ||_{L^2} 
\end{equation}

To do this, we note that the spectrum of $A + R_{-1}^{\#}$ consists of bands of the form $\lambda_k^j = k + \mu_k^j$ where $\abs{\mu^j_k} = O(k^{-1})$. Hence for $k$ sufficiently large, the entire band is contained in a ball of radius $\frac{1}{4}$ around $k$. Let $\gamma_k$ be a circle of radius $\frac{1}{2}$ centered at $k \in \mathbb{N}$. Then for $k$ sufficiently large,

\begin{equation}
\pi_k = \frac{1}{2\pi i} \int_{\gamma_k} (\lambda - (A + R_{-1}^{\#}))^{-1} \, d\lambda
\end{equation}

and

\begin{equation}
\pi'_k = \frac{1}{2\pi i} \int_{\gamma_k} (\lambda - (A + R_{-1}^{\#} + R_{-\infty} ))^{-1} \, d\lambda
\end{equation}

Hence 

\begin{equation}
(A + R_{-1}^{\#})^N( \pi_k \pi'_k - \pi'_k) = \frac{1}{2\pi i} \int_{\gamma_{k}} (\lambda - (A + R_{-1}^{\#}))^{-1}(A + R_{-1}^{\#})^N R_{-\infty} \pi'_k (\lambda - (A + R_{-1}^{\#} + R_{-\infty}))^{-1} \, d\lambda
\end{equation}

For $\lambda \in \gamma_k$, the distance between $\lambda$ and the spectrum of both $A + R_{-1}^{\#}$ and $A + R_{-1}^{\#} + R_{-\infty}$ is bounded below by $\frac{1}{4}$. Hence the norms of both resolvents are bounded by $4$, which implies the norm of the left hand side is bounded by $2 ||(A + R_{-1}^{\#})^N R_{-\infty}\pi'_k ||_{L^2}$. Now suppose that we choose $k \geq k_0$ so that the above estimate holds. Then, repeating the argument on p. 255 of \cite{Gu} we build a sequence of unitary operators $L_k : V'_k \to V_k$. Since $L_k$ is a function of $P_k$ and $A$ commutes with $D_\theta$, each $L_k$ does as well. Define the unitary operator $L$ by declaring $L = L_k$ on $V'_k$ for $k \geq k_0$ sufficiently large so that the above estimate holds. To define $L$ on $\bigoplus_{1 \leq k \leq k_0} V'_k$, let $U_k$ denote the eigenspace of $\widehat{I}_2$ of eigenvalue $k$. and let $\varphi^k_m$ be an orthonormal basis of $U_k$ consisting of joint eigenfunctions of $D_\theta$. Then $W\varphi^k_m$ is a basis of $V'_k$ which are also joint eigenfunctions of $D_\theta$. Define $L$ by taking $W\varphi^k_m$ to the corresponding standard spherical harmonic of joint eigenvalue $(k,m)$. $L$ clearly commutes with $D_\theta$ as well as $A$. Also, by construction $L(A + R_{-1}^{\#} + R_{-\infty})L^* = A + L(R^{\#}_{-1} + R_{-\infty})L^*$ preserves each $V_k$ eigenspace, so commutes with $A$. This implies that $L(R^{\#}_{-1} + R_{-\infty})L^* = R^{\#}$ commutes with $A$. Finally the estimate above is used to prove that $L - I$ is a smoothing operator in the same way as in \cite{Gu}.

\end{proof}

\begin{Prop} \label{prop:nofiniterank}
Suppose that $\text{Spec}(A + R_{-1}^{\#}) = \text{Spec}(A) = \mathbb{N}$ where $R_{-1}^{\#} \in \Psi^{-1}$ is self-adjoint and commutes with $A$. Then $R_{-1}^{\#}$ is a finite rank operator.
\end{Prop}

\begin{proof}
Since $R^{\#}$ commutes with $A$, we can choose an orthonormal basis of $V_k$, $e^k_j$ satisfying $R^{\#}e^k_j = \mu^k_j e_j$. Since $R^{\#} \in \Psi^{-1}$, we have $\abs{\mu^k_j} = O(k^{-1})$. The fact that $\text{Spec}(A + R^{\#}) = \mathbb{N}$ implies that for $k$ large, $R^{\#}|_{V_k} = 0$ which shows that $R^{\#}$ is finite rank.

\end{proof}

\end{document}